\documentclass[11pt]{amsart}

\usepackage{amsmath,color}
\usepackage{amsthm}
\usepackage{amssymb}
\usepackage{graphicx}
\usepackage{hyperref}

\numberwithin{equation}{section}

\definecolor{darkgreen}{rgb}{0,0.55,0}

\newcommand{\chara}{1\!\!1}

\newcommand{\wt}{\widetilde}

\newcommand{\e}{\varepsilon}

\newcommand{\vp}{\varphi}

\newcommand{\R}{{\mathbb R}}

\newcommand{\sign}{\operatorname {sign}}
\newcommand{\yt}{{y^\tau}}
\newcommand{\yn}{{y^n}}

\newcommand{\rr}{r}
\newcommand{\divf}{\operatorname{div}}

\newcommand{\beq}{\begin{equation}}
\newcommand{\eeq}{\end{equation}}

\newtheorem{theorem}{Theorem}[section]
\newtheorem{lemma}[theorem]{Lemma}
\newtheorem{proposition}[theorem]{Proposition}

\newtheorem{remark}[theorem]{Remark}

\begin{document}

\author{ Bernardo Galv\~ao-Sousa \and Robert L. Jerrard }
\address{Department of Mathematics, University of Toronto, Toronto, Ontario, Canada}\email{beni@math.toronto.edu}
\address{Department of Mathematics, University of Toronto, Toronto, Ontario, Canada}\email{rjerrard@math.toronto.edu}

\thanks{The second author was partially supported by the National Science and
Engineering Research Council of Canada under operating Grant 261955.}

\title[Accelerating fronts in semilinear wave equations]
{Accelerating fronts in semilinear wave equations}

\begin{abstract} 
We study dynamics of interfaces in solutions of the equation
$
\e \Box u + \frac 1 \e f_\e(u)=0$,
for $f_\e$ of the form $f_\e(u) = (u^2-1)(2u- \e\kappa)$, for $\kappa\in \R$,
as well as more general, but qualitatively similar, nonlinearities.
We prove that for suitable initial data, solutions exhibit 
interfaces that sweep out  timelike hypersurfaces of mean curvature
proportional to
$\kappa$. In particular,  in $1$ dimension these interfaces 
behave like a relativistic point particle subject to
constant acceleration.
\end{abstract}
\maketitle

\vspace{.2in}

\section{introduction}

In this paper we consider the dynamics of interfaces in semilinear hyperbolic equations. The simplest example that we study is the equation
\beq
\e( u_{tt} - u_{xx}) + \frac 1 \e(u^2-1)(2u - \e \kappa) = 0,
\label{nlw2}\eeq
where we assume for concreteness that $\kappa>0$.
Here the nonlinearity $f_\e(u) = (u^2-1)(2u - \e \kappa)$ has the form $f_\e = F_\e'$, where $F_\e$
has local minima at $u=\pm 1$, with $F_\e( \pm1) = \pm \frac 23  \e \kappa$. 
Thus  the state $u=-1$ has slightly lower energy than the state $u=1$, and one might expect that there exist solutions in which the low-energy phase $u=-1$ grows at the expense of the higher-energy phase.
This is what we prove. In fact we show that, for suitable initial data, solutions exhibit an interface that
behaves like a relativistic mass subject to constant acceleration proportional to the parameter $\kappa$.
Equivalently, the interface sweeps out a timelike curve of constant {\em Minkowskian} curvature, proportional to  $\kappa$, in the $(t,x)$-plane.

It turns out that our analysis extends with rather
 few changes to 
wave equations on suitable Lorentzian manifolds $(N, {\bf h})$. Thus we will also consider the equation
\beq
\e \Box_{\bf h} u + \frac  1 \e f_0(u)  +  \kappa f_1(u)= 0 
\label{nlw3}\eeq
where $\Box_{\bf h}$ is the Laplace-Beltrami (wave) operator on $(N, {\bf h})$,
$\kappa$ is a smooth function, and $\frac 1\e f_0(u) + \kappa f_1(u)$ generalizes
the nonlinearity in \eqref{nlw2} in a natural way. 
In this situation, analogous to \eqref{nlw2}, we show that for well-prepared data, interfaces sweep out 
timelike hypersurfaces of {\em prescribed } mean curvature $\kappa$, with respect
to the Lorentzian metric ${\bf h}$.

In the case when $(N, {\bf h})$ is just $1+n$-dimensional Minkowski space
and $\kappa\equiv 0$, corresponding to the situation when the two potential wells have equal depth, similar results were proved by the first author in \cite{j-mms}, following partial
results of \cite{bno}. 
Thus, the present paper consists of a number of improvements of the basic argument
developed in \cite{j-mms}:  we extend the results to the case $\kappa\ne 0$, 
we show that they remain valid
on Lorentzian manifolds more general than Minkowski space, and we
drop some convenient but artificial restrictions imposed in \cite{j-mms} on the topological type of 
the hypersurfaces considered. A key point in our analysis is that 
if $\kappa$ is a nonzero constant, then in certain weighted energy estimates 
it is much more useful to use, not the canonical conserved energy associated to
the actual equation \eqref{nlw3} under study, but rather the conserved energy {\em associated to the
$\kappa = 0$ equation}. (See Remark \ref{keypoint}.)
This  simple observation plays a crucial role in our arguments and makes the extension 
of techniques developed in \cite{j-mms} to the more general situation considered here surprisingly straightforward.

\smallskip

Equations such as \eqref{nlw3}, with $\kappa \ne 0$, have been studied in the cosmological literature as  models
for what is called the decay of a false vacuum. 
This arises from models in which the 
universe is described by a quantum field theory for which an equation
like \eqref{nlw3} (or a more complicated but in some ways similar equation) is a low-energy limit, and whose state is initially given by a constant function $u \equiv v_f$, where $v_f$ is a ``false vacuum": a local, {\em but not global}, minimum of some underlying potential function. In the example \eqref{nlw2},
if $\kappa>0$ then $v_f = 1$, and the ``true vacuum", or global minimizer
of the potential function $F_\e$, is   $v_t =-1$.
In this situation, a quantum tunnelling event could in principle lead to the nucleation of region in which $u = v_t$. This scenario was investigated  in a series of papers by Coleman and coworkers;
see for example \cite{Coleman77}, which estimates via a formal semiclassical approximation the probability per unit time per unit volume of such a tunneling event.
Our results have nothing to say about this, but describe the dynamics of a fully-formed interface
between false and true vacuums in a universe governed by \eqref{nlw3}, showing that if
the interface has an energetically optimal structure, then it behaves like a hypersurface
of constant Lorentzian mean curvature proportional to the difference in energy between the true and false vacuums. 


Earlier work on dynamics of energy concentration sets in hyperbolic
equations includes \cite{j-wave, l-wave, gs, stu}, all of which 
consider situations in which energy concentrates around points
rather than submanifolds, as in \cite{bno, j-mms} and the present paper.
The dynamics of interfaces in equations such as
\eqref{nlw3} is studied from a formal point of view in \cite{rn}. 

A lengthy discussion of related elliptic and parabolic results,
with a heavy bias toward the $\kappa=0$ case, 
is contained in \cite{j-mms}. 
For $\kappa \ne 0$,
there is a rather strong analogy between the phenomena we study
and propagating fronts in semilinear parabolic equations, a subject that
has attracted a great deal of study, dating back to the 1930s \cite{kpp}. In particular,  the problem that 
formally determines the profile and (relativistic) acceleration of interfaces
(see \eqref{trwave1} or \eqref{trwave3})
is exactly the same one that determines the profile and
velocity of propagating fronts in certain parabolic problems,
see for example \cite{aw, fm, bss}.
There is also an analogy between our work and
results that establish an asymptotic connection
between elliptic analogs of \eqref{nlw3} and
surfaces of prescribed Euclidean (or more generally Riemanian) 
mean curvature, see for example \cite{pr,ht}.

\smallskip
This paper is organized as follows: In order to highlight some main ideas with as few preliminaries as possible,
we consider in Section \ref{S:simple}  the case \eqref{nlw2} of an equation in one space  dimension associated with two potential wells of unequal depth. This discussion isincluded just to illustrate our arguments in a simple setting, and is not needed in
later pars of the paper.

We therefore defer until Section  \ref{S:statement}
both the statement of our main result, and the introduction
of some notation that is used throughout the rest of the paper.
In Section \ref{S:adapted} we introduce a coordinate system in
which many of our main estimates will take place, stating the
properties that we will need and deferring most proofs to
Section  \ref{S:Phi}. The heart of our argument consists
of weighted energy estimates in this adapted coordinate system.
These are carried out in Section \ref{S:weenc}. In Section
\ref{S:pot3.2}, these estimates are combined with rather standard
energy estimates away from the interface in a iterative
argument that completes the proof of our main theorem.


\section{the simplest nontrivial equation}\label{S:simple}

In this section we consider the $1$-dimensional equation \eqref{nlw2}.
All the  results  in this section are essentially subsumed in Theorem \ref{T.mink},
and most of the main ideas in Theorem \ref{T.mink}
appear here, in somewhat simpler form.

It is convenient to consider initial data such that at $t=0$, 
\beq
(u, u_t) = (-1, 0)\mbox{ for all $x$ near 0},\quad\quad
(u,u_t) = (1,0)
\mbox{ for all $x \ge R$}
\label{1ddata1}\eeq
for some $R$. (More conditions on the data will be imposed later.)
Noting that the constant functions 
$\pm 1$ are both  solutions of \eqref{nlw2},
standard facts about finite propagation speed for solutions of \eqref{nlw2} imply that
\beq
u(t,x) = -1\mbox{ for $(t,x)$ near }(0,0), \quad\quad
u(t,x) = 1\mbox{ for }x \ge R + |t|.
\label{1dgbc}\eeq

\subsection{change of variables}
As suggested above, one might guess that for suitable initial data, solutions will exhibit an interface that sweeps
out a timelike curve of constant (nonzero) Minkowskian curvature proportional to the parameter $\kappa$ that controls to the difference in depth of the two energy wells. Such curves have
the form $\{ (t,x) : x^2 - t^2 = c, x>0\}$ modulo translations and reflections. We thus start by changing variables in such a way as to ``straighten out"   a $1$-parameter family of such curves. Thus, 
we introduce new coordinates $(\theta,\rr)$ defined (for $\theta\in \R, \rr>0$) by 
\beq
(t,x) = (\rr\sinh \theta ,\rr\cosh \theta ) = \psi(\theta, \rr) \in \{\, (t,x) :  |t| < x \}. 
\label{psi.def}\eeq
These are just Minkowskian polar coordinates, with $\theta$ being the angular and $\rr$ the radial coordinate. Note that every coordinate line $\rr = r_0$ is a timelike curve of constant curvature $\frac 1{r_0}$ with respect to the Minkowski metric, which in these coordinates takes the form $ds^2 = -\rr^2 d\theta^2 + d\rr^2$. 
We will treat $\theta$ as a timelike coordinate, and $\rr$ as  spacelike. 

If $u$ solves \eqref{nlw2} and we write $v = u \circ \psi$, then we find that $v$ satisfies
\beq
\e\left( \frac 1 {\rr^2} v_{\theta \theta}  -  v_{\rr \rr}  - \frac 1 {\rr}  v_{\rr} \right)   
+ \frac 1 \e(v^2-1)(2v - \e \kappa) =0,\quad\quad\theta\in \R, \rr>0.
\label{nlwpolar}\eeq
If we imagine that $ v_{\theta \theta} \approx 0$ and  that $\frac 1{\rr} \approx c$ constant, then
this looks like the equation
\beq
\e(-q'' - c q') +  \frac 1 \e (q^2-1)(2q- \e\kappa) \ = \ 0
\label{trwave1}\eeq
for the profile $q$ and wave speed $c$ of  traveling wave solutions of the parabolic counterpart of \eqref{nlw2}. 
This is known to have the $1$-parameter family of solutions
\beq
c = \kappa, \quad q= \tanh(\frac {\rr - r_0} \e),\quad\quad\quad r_0\in \R.
\label{trwave2}\eeq
(We have set things up so that the profile $q$ is independent of the parameter $\kappa$.) 
Note that if we choose $r_0 = \frac 1 \kappa$, then all the nontrivial behavior of 
$ q_\e(r) := \tanh(\frac 1 \e(\rr - \frac 1\kappa))$ is concentrated in an $\epsilon$- neighborhood of 
$r = \frac 1\kappa$, which is consistent with the heuristic  $\frac 1 {\rr}\approx \kappa=c$.

Thus, we will study \eqref{nlwpolar} with initial data such that
\begin{equation}
v( 0, \rr) \approx \tanh(\frac 1 \e(\rr -r_0)), \quad\quad
v_{\theta}( 0,\rr) \approx 0
\quad\quad\mbox{ for  $ \rr>0$}
\label{data.vague}\end{equation} 
where  henceforth we set
\beq\label{choose.r0}
r_0 := \frac 1\kappa.
\eeq
 Indeed, we will show that for data of this form, solutions are
approximately independent of $\theta$, and hence remain concentrated about the curve $r = \frac 1\kappa$.
Recall also that we are assuming \eqref{1ddata1}, which in the new variables implies that 
for $\theta = 0$, 
\begin{equation}\label{1ddata2}
\mbox{
 $(v, v_\theta) = (-1,0)$ for all $r$ close to $0$, 
\quad\quad
 $(v, v_\theta) = (1,0)$ for  $r\ge R$},
 \end{equation}
 and this implies \eqref{1dgbc}, which translates to
\begin{align}
v(\theta, r) = -1 \mbox{ for $r$ near $0$, }
\quad\quad
v(\theta, r) = 1 \mbox{ for  }r\ge Re^{|\theta|}.
\label{rbc}
\end{align}
for every $\theta\in \R$.

\subsection{differential energy inequality}

We next define
\[
e_\e(v) := \frac \e 2 (\frac {v_{\theta}^2}{\rr^2} + v_{\rr}^2) +  \frac 1 {2\e} (v^2-1)^2.
\]
A short computation shows that if $v$ is a sufficiently smooth solution of \eqref{nlwpolar}, then
\beq
\frac d{d\theta}e_\e(v) = \e( v_{\theta} v_{\rr})_{\rr} + \mbox{Term 1} + \mbox {Term 2}
\label{polar.energy}\eeq
where
\[
\mbox{Term 1} = \e \frac {v_{\theta}}{\rr} v_{\rr} (1- \kappa \rr),\quad\quad
\mbox{Term 2} = \kappa \e v_{\theta} ( v_{\rr} -  \frac 1 \e(1-v^2)).
\]
Term 1 should be small in $L^1$ if $v_\theta$ is small and 
$v_{\rr}$ is concentrated near $\rr = \frac 1\kappa$.
Also, the profile $q_\e(\rr) = \tanh(\frac 1\e(\rr-  \frac 1 \kappa))$ satisfies 
$\partial_{\rr}q_{ \e} =  \frac 1 \e(1-q_\e^2) $, so that if $v \approx q_\e$ in a sufficiently strong
sense, then Term 2 should be small.

\begin{remark}
\label{keypoint}
Note that \eqref{nlwpolar} has an {\em exactly} conserved energy: a computation shows that
\[
\frac d{d\theta} \left( \frac \e{2r}v_\theta^2 + \frac {\e r} 2 v_r^2 + \frac 1 \e F_\e(v)
\right)
=  ( r v_r v_\theta)_r, \  \ \mbox{ where }
F_\e'(s) = (s^2-1)(2s- \kappa\e).
\]
It turns out that it is much more useful to work  with the approximately conserved energy $e_\e(v)$
defined above. 
This observation, although very simple, is a key point in our analysis.
\end{remark}

\subsection{lower energy bound}

Note that if $v$ satisfies \eqref{rbc}, 
\begin{align}
\int_0^\infty  \frac \e 2 v_{\rr}^2 +  \frac 1 {2\e} (v^2-1)^2 
 dr 
& \ \ge \ 
\int_0^\infty  |(1-v^2)v_{\rr}|  \nonumber \\
&\ \ge \  
\left| \int_0^\infty  ( v - \frac 13 v^3)_r  dr\right| \ \nonumber\\
&= \  \frac 43 =: c_0.
\label{lbd}\end{align}

\subsection{weighted energy estimates in new variables}

Next, given a solution $v$ of \eqref{nlwpolar}, we 
will write
\beq
\zeta_1(\theta) =\left. \int_0^\infty [1+ (r-r_0)^2] e_\e(v) \ dr \right|_\theta- c_0 .
\label{eta1.def1}\eeq
For the initial data we consider, \eqref{rbc} holds,  and  then \eqref{lbd} implies that
\beq
\zeta_1(\theta) \ge  \zeta_2(\theta) :=  \int_0^\infty \frac \e 2 (\frac {v_\theta}{\rr})^2 + (r-r_0)^2\left.\left( \frac \e 2 v_{\rr}^2 +  \frac 1 {2\e} (v^2-1)^2 \right)  \ dr\right|_\theta .
\label{eta12}\eeq
Using \eqref{polar.energy}, we compute
\begin{align*}
\zeta_1'(\theta)
&= 
\int_0^\infty[1+(r-r_0)^2] \left[ \e( v_{\theta} v_{\rr})_{\rr} + \mbox{Term 1} + \mbox{Term 2}\right] dr.
\end{align*}
Every term in the integrand contains a factor of $v_\theta$, which due to \eqref{rbc} has compact
support in $(0, \infty)$, so the integral clearly exists, and we can integrate by parts without problems.
It also follows from \eqref{rbc} that $1 \le R e^{|\theta|}  \frac 1r$ on the support of $v_\theta$. Thus
\begin{align*}
\int_0^\infty[1+ (r-r_0)^2]\e( v_{\theta} v_{\rr})_{\rr}  & = \ 
-\int_0^\infty 2 \e v_{\theta} (r-r_0) v_{\rr} \\
&\le
R e^{|\theta|} \int_0^\infty 2 \e \frac{| v_{\theta}|}r \, |r-r_0|\, |v_{\rr}|  \le 2 R e^{|\theta|} \zeta_2(\theta).
\end{align*}
Recalling that $\kappa = r_0^{-1}$, 
elementary estimates  yield
\[
|\mbox{Term 1}| \le 
\frac { \e\kappa} 2\left[ (\frac {v_\theta}{\rr})^2 + (r-r_0)^2  v_{\rr}^2\right] , 
\]
and
\[
|\mbox{Term 2}| \le 
\frac { r \e\kappa} 2\left[ (\frac {v_\theta}{\rr})^2 +(v_\rr - \frac 1 \e(1-v^2))^2\right] , 
\]
Repeatedly using \eqref{rbc} to bound $r$, and recalling that $\zeta_2 \le \zeta_1$,
we deduce that
\[
\zeta_1'(\theta) \le C Re^{2|\theta|} \zeta_1(\theta) + CR^3 e^{3|\theta|}\int_0^\infty \frac \e 2 (v_\rr - \frac 1 \e(1-v^2))^2.
\]
However, arguing as in \eqref{lbd},
\begin{align}
\int_0^\infty \frac \e 2 (v_\rr - \frac 1 \e(1-v^2))^2
&=
\int_0^\infty \frac \e 2 v_\rr^2 + \frac 1 \e(1-v^2)^2
-
\int_0^\infty (1-v^2)v_\rr\nonumber \\
&\overset{\eqref{rbc}
}=
\int_0^\infty \frac \e 2 v_\rr^2 + \frac 1 \e(1-v^2)^2
-
c_0\nonumber\\
&\le \zeta_1(\theta).
\label{qest}\end{align}
We conclude that
\begin{equation}
\zeta_1'(\theta) \le C(R+R^3) e^{3|\theta|} \zeta_1(\theta).
\label{1dGr}\end{equation}

\subsection{conclusions about $v$}
At this point, we have proved most of the following proposition.

\begin{proposition}
Let $v$ solve \eqref{nlwpolar} with initial data satisfying 
\eqref{1ddata2}. Let $r_0 = \kappa^{-1}$, and define
$\zeta_1,\zeta_2$ as in  \eqref{eta1.def1}, \eqref{eta12}.
Then there exists a constant $C$, depending on the parameter $R$ in \eqref{1ddata2},
such that for every $\theta\in \R$,
\begin{equation}
\zeta_2(\theta) \le \zeta_1(\theta) \le e^{C(e^{3|\theta|} - 1)} \zeta_1(0) \quad\mbox{ for every }\theta,
\label{eta1est}\end{equation}
As a result, 
\begin{equation}\label{v.L2}
\int_0^\infty \big| v(\theta,r)-v(0,r)\big|^2 \, \frac{dr}{r^2} 
		\le C(\theta,R) \frac { \zeta_1(0)}{\e}, \quad\qquad C(\theta,R) := C |\theta| \int_0^\theta e^{C(R)(e^{3|s|}-1)} \, ds .\end{equation} 
In particular, there exists initial data for which the solution $v$ satisfies
\beq
\int_0^\infty |v(\theta,r) - \sign(r-r_0)|^2 \frac {dr}{r^2 } \le C(\theta,R)\e \quad\mbox{ for all }\theta.
\label{L2forgooddata}\eeq
\label{prop.1dpolar}\end{proposition}

\begin{remark}\label{r:extra}
For any $\delta>0$,
there exists $C = C(\delta, r_0)$ such that
\beq\label{notsharp}
\int_{0}^\infty |\tanh(\frac{r-r_0}\e)-\sign(r-r_0)|^2\frac {dr}{\max(r,\delta)^2}\le C\e.
\eeq
So \eqref{L2forgooddata} implies that
\beq
\int_{0}^\infty |v(\theta,r) - \tanh(\frac{r-r_0}\e)|^2 \frac {dr}{\max(r,\delta)^2 } \le C\e.
\label{L2gooddata2}\eeq
However, \eqref{notsharp} says precisely that \eqref{L2forgooddata} is not a sharp enough estimate to
determine the profile of $v$,
so it  would arguably be a little misleading to insist on $\tanh$ rather than
$\sign$ in estimates such as  \eqref{L2gooddata2},
\eqref{L2forgooddata}.

On the other hand, standard spectral estimates  
imply that for every $\theta$,  there exists some
$r_\e(\theta)$ such that the solution
$v$ in \eqref{L2forgooddata} satisfies
\[
\int_0^\infty |v(\theta, r) - \tanh (\frac {r-r_\e(\theta)}{\e})|^2 \ \frac{dr}\e \le C \zeta_1(\theta)
\le C(\theta)\e^2, 
\]
and then it follows from \eqref{L2forgooddata} that $|r_\e(\theta)- r_0|\le C \e$.
So although it is not captured in \eqref{L2forgooddata}, 
our estimates
do in fact show that $v$ is close to a scaled, translated hyperbolic tangent.
\end{remark}

\begin{proof}To complete the proof, notice first that  \eqref{eta1est} follows
from \eqref{1dGr} and \eqref{eta12}, via a form of Gr\"onwall's inequality.
Next, for every $r$,
\[
\frac 1 {r^2}\Big( v(\theta,r)-v(0,r)\Big)^2 \ \le \ |\theta| \int_0^\theta  \frac 1 {r^2} v_\theta^2(s,r)  ds . 
\]
We deduce \eqref{v.L2}  by integrating this inequality with respect to $r$
then using Fubini's Theorem and  \eqref{eta1est}.
Finally, to prove \eqref{L2forgooddata}, it now suffices to exhibit initial data $(v, v_\theta)|_{\theta=0} $
satisfying
\eqref{1ddata2}, and such that
\beq
\zeta_1(0) \le C \e^2, \quad\quad \int_0^\infty |v(0,r)-\sign(r-r_0)|^2 \frac {dr}{r^2}\le C \e^2.
\label{gooddata2}\eeq
To do this, let $v_\theta|_{\theta=0} = 0$, and let
\[
v(0,r) := \bar q_{\e, r_0}(r-r_0),
\]
where
\beq\label{barqer}
\bar q_{\e,r_0}(s) = \chi_{r_0}(s) \tanh(\frac s \e) + (1-\chi_{r_0}(s))\sign(s),
\eeq
and for $r>0$ we define $\chi_r\in C^\infty_c(\R)$ to be a  function such that 
\beq
\mbox{$\chi_r(s) = 1$ if $|s|\le \frac r3$ and
$\chi_r(s) = 0$ if $|s|\ge \frac {2r}3$,}
\qquad
\quad\qquad
|\chi_r'|\le C/r
\label{chi.def}\eeq 
It is straightforward to verify
that this initial data satisfies \eqref{gooddata2}.
\end{proof}

\subsection{conclusions about $u$.}

Proposition \ref{prop.1dpolar} yields uniform estimates for $v$ on sets
of the form $\{ (r, \theta) : r>0, |\theta| <\Theta\}$,
corresponding to uniform estimates of the original solution $u$
in a sector $\{ (t,x) : x>0 ,  |t| < x \tanh \Theta\}$. (Recall that $u = v \circ \psi$ for
$\psi$ defined in \eqref{psi.def}.) We next show
infer estimates of $u$ in a 
spacetime slab $(-T,T)\times \R$.

\begin{proposition}\label{1dstandard}
Fix $\e\in (0,1]$ and let $u$ solve 
\eqref{nlw2} with initial 
data $u(0,x) = \bar q_{\e,r_0}(x - r_0)$, $u_t(0,x)=0$,
where $\bar q_{\e,r_0}$ is defined in \eqref{barqer} above and $r_0=\kappa^{-1}$.

For every $T>0$,
there is a constant $C(T)$, independent of $\e$, such that
\beq
\int_{-T}^T \left|u(t,x) -  \sign(x-\gamma(t))
 \right|^2 \ dx\, dt \ \le \  C(T) \e.
\label{u.L2}\eeq
where 
$\gamma(t) = ( r_0^2 + t^2)^{1/2}$. 
\end{proposition}

\begin{remark}
Although we have stated here only an analog of \eqref{L2forgooddata},
our arguments also establish an analog 
of \eqref{eta1est}, i.e., energy estimates 
showing that, for
a large class of initial data,  
energy concentrates
near the curve $(t,\gamma(t))$. We already know this, modulo  
a change of variables, in the sector
$\{ (t,x) : x>0, |t|\le  x\tanh \Theta\}$ so the new point is
energy estimates outside this sector,
which are established in 
\eqref{u.energy1} below.
\end{remark}

\begin{proof}
Fix $T>0$ and let $\Theta$ be such that $T = \frac{r_0}{3} \sinh \Theta$, i.e., $\Theta = \sinh^{-1} \frac{3T}{r_0}$. We will start by considering a solution $u$ for initial data satisfying
\beq
(u, u_t)(0,x) = (-1, 0) \mbox{ for all $x < \delta$,}\qquad
(u, u_t)(0,x) = (1, 0) \mbox{ for all $x  > R$,}
\label{udata}\eeq
for some $0<\delta < r_0 < R$.
We will only specialize later to the initial data in the statement of
the Proposition. Until we do so,
all constants in our argument may depend on $r_0 = \frac 1\kappa$, $\Theta$ (hence on $T$), and the parameter $R$ above.

The choice \eqref{udata} of initial data implies that Proposition \ref{prop.1dpolar} 
applies to $v = u \circ  \psi^{-1}$.

We will write 
\[
x_0(t) :=  \sqrt{t^2+ \frac 19 r_0^2}, \qquad
x_1(t) :=  \sqrt{t^2+ \frac 49 r_0^2}.
\]
The idea is simply that  results from Proposition \ref{prop.1dpolar} 
imply that $u \approx -1$ (with respect to the $H^1$ norm) in the 
set $\{(t,x) : x_0(t) < x < x_1(t) \}$, see the shaded region in Figure \ref{fig:1}.
We combine this with the fact  that  $(u,u_t)|_{t=0} \approx (-1,0)$ 
to argue that $u \approx  -1$ in $H^1$ in the entire
set $\{ (t,x) : |t|< T, x  < x_1(t) \}$.  

\begin{figure}[h]
\centering
\includegraphics*[width=300pt]{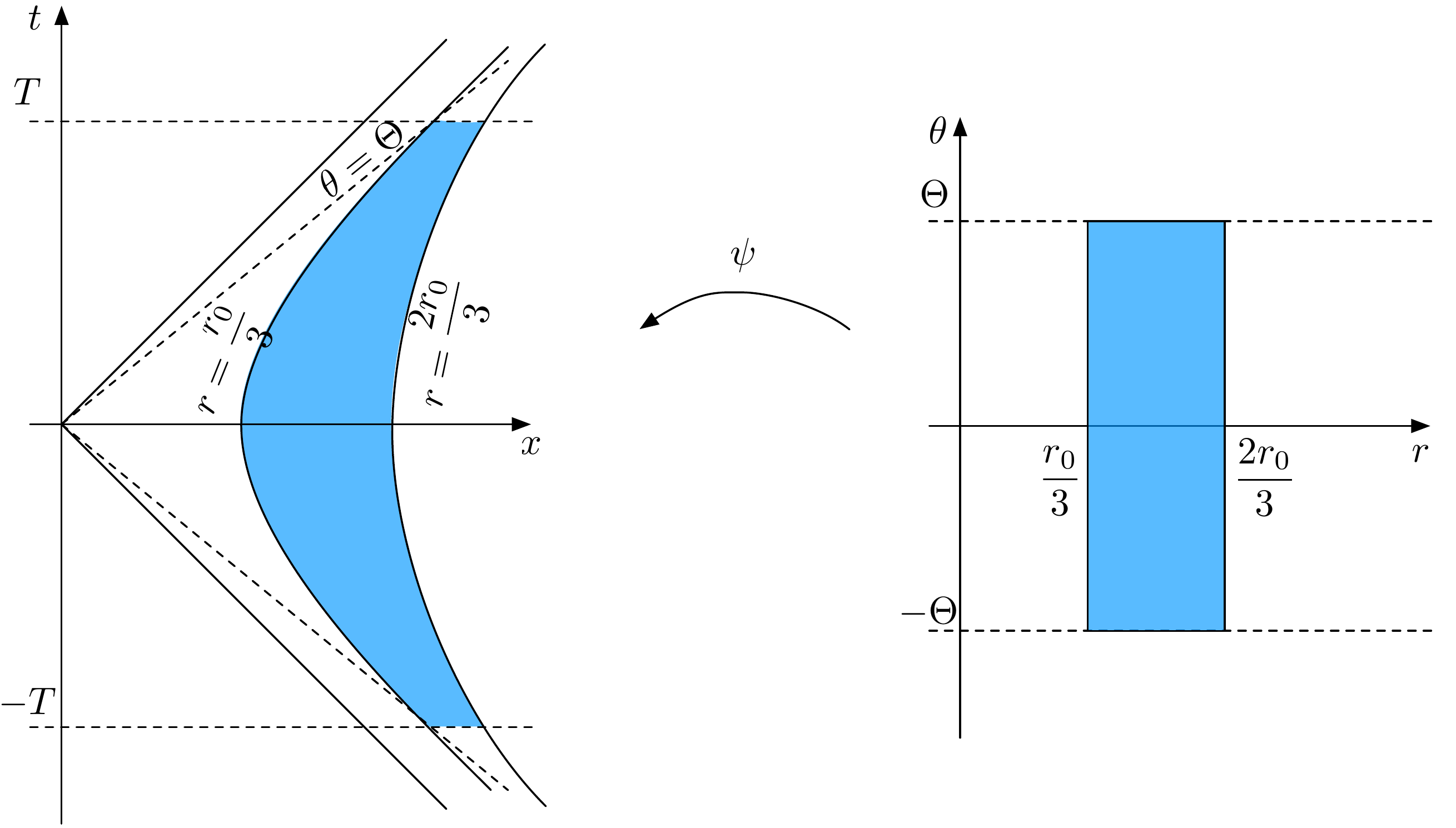}
\caption{}
\label{fig:1}
\end{figure}

{\bf Step 1}. Indeed, let $\chi$ be a smooth  function such that $0\le \chi \le 1$ and
\begin{align*}
\chi(x,t)  = 1 \mbox{ if }  x \le x_0(t),
\qquad  
\chi(x,t) & = 0 \mbox{ if }  x \ge x_1(t) .
\end{align*}
We will  write $e_\e^o(u)$ for an energy density in the original coordinates
defined by
$$
e^o_\e(u) = \frac\e2\big( u_t^2 + u_x^2 \big) + \frac1{2\e} (1-u^2)^2.
$$
Then rather standard energy arguments, which we recall in Step 3 below,
imply that
\begin{align}
\int \chi e^o_\e(u) \bigg|_t
	& \leq  e^{\kappa t} \int \chi e^o_\e(u) \bigg|_0 + C_1 \int_0^t  \int_{x_0(t)}^{x_1(t)}
	e^{\kappa(t-s)} e^o_\e(u) \, dx\, ds.
\label{est_eu}
\end{align}
It is straightforward to check that there exists some $C$ such that
$e^0_\e(u)\circ \psi \le C e_\e(v)$ in $[-\Theta, \Theta] \times [\frac 12 r_0,. \frac 23 r_0]$.
On the same set, the Jacobian determinant  $\det (D\psi)$  is bounded 
and 
\[
\frac 1C e_\e(v) \le  \frac \e 2 (\frac {v_\theta}{\rr})^2 + (r-r_0)^2\left( \frac \e 2 v_{\rr}^2 +  \frac 1 {2\e} (v^2-1)^2 \right).
 \]
Hence by a change of variables,
\[
\int_0^t  \int_{x_0(t)}^{x_1(t)} e^o_\e(u) \, dx\, ds 
\le
C\int_0^\Theta \int_{r_0/3}^{2r_0/3}  e_\e(v) \ dr \ d\theta
\le 
C\int_0^\Theta \zeta_2(\theta) d\theta.
\]
We combine this with \eqref{est_eu} and using  \eqref{eta1est}, and note
that $\int\chi(t,x) e^0_\e(u)(0,x)\ dx \le C \zeta_1(0)$, as a result
of \eqref{1ddata1}. These computations lead to the inequality
\beq
\int_{\{ (t,x) : |t|\le T,\, x<x_0(t) \} } e_\e^o(u) \ dx\,dt \ \le C\zeta_1(0).
\label{u.energy1}\eeq

{\bf Step 2}. 
We now use the above estimates to prove \eqref{u.L2}
for $u$ solving \eqref{nlw2} with initial data satisfying 
\eqref{udata} together with
\beq
\zeta_1(0) \le C\e^2 \ \  \mbox{ for }v = u\circ \psi^{-1},
\qquad
\int_R |u(0,x) - \sign(x-r_0)|^2 \le C \e.
\label{udata2}\eeq
In particular, these conditions are satisfied by the specific data described in the
statement of the proposition, exactly as in the proof of Proposition \ref{prop.1dpolar}.
We will write 
$Z(t,y) := (t, y + x_0(t))$. 
Then for  for every $y$, 
\[
u(Z(t,y)) - u(Z(0,y)) = \int_0^t u_t(Z(s,y)) + x_0'(s)u_x(Z(s,y)) \ ds \le    C \left( t \int_0^t  \frac{e_\e^o(u)(Z(s,y)) } \e\ ds\right)^{1/2}.
\]
Also, $\{ (t,x) : |t|\le T,\, x<x_0(t) \} = \{ Z(t,y) :|t|\le T, y\le 0 \}$, 
so by squaring the above inequality, integrating 
from $y=-\infty$ to $y=0$,  integrating in $t$,
changing variables, and using \eqref{u.energy1} and \eqref{udata2}, we find that
\beq
\int_{-T}^T\int_{-\infty}^{x_0(t)} |u(t,x) + 1|^2 \, dx \ dt \ \le \  \frac  C \e \zeta_1(0) \le C \e.
\label{u.L2a}\eeq

Next, recall that $u\circ \psi = v$, 
so by a change of variables (see \eqref{rbc}) and \eqref{L2forgooddata},
\[
\int_{-T}^T\int_{x_0(t)}^{R+|t|} |u - \sign(x-\gamma(t))|^2 
\le 
C \int_{-\Theta}^\Theta \int_{r_0/3}^{R e^{|\theta|}} | v(\theta, r)- \sign(r-r_0)|^2\, dr\, d\theta\  \le C \e.
\]
Recalling that $u(t,x) \equiv 1$ for $x \ge R+|t|$, we deduce that if \eqref{udata}, \eqref{udata2}
hold then
\beq
\int_{-T}^T\int_\R  |u (t,x) - \sign(x-\gamma(t))|^2  \le C \e.
\label{tildeU.L2}\eeq

{\bf Step 3: proof of \eqref{est_eu}:}
A short calculation shows that 
\begin{equation}
\frac{d}{dt} e^o_\e(u) = \e (u_tu_x)_x + \kappa u_t (1-u^2).
\label{eoriginal}\end{equation}
Then it follows from \eqref{eoriginal} that 
\[
\frac{d}{dt} \int \chi e^o_e(u)
	 = \int \chi_t e^o_\e + \e \int \chi (u_t u_x)_x + \kappa \int \chi u_t (1-u^2).
\] 
We integrate by parts and use some elementary inequalities to find
\[
\frac{d}{dt} \int \chi e^o_e(u)
	 \leq \int ( |\chi_t|+ |\chi_x|) e^o_\e + \kappa\int \chi e^o_\e .
\] 
This implies that
$$
\frac{d}{dt} \left( e^{-\kappa t} \int \chi e^o_\e(u) \right)
	\leq e^{-\kappa t} \int \big( |\chi_t| + |\chi_x| \big) e^o_\e(u)
	\leq C_1 e^{-\kappa t} \int_{{\rm supp}\, \chi_x} e^o_\e(u),
$$
for $|t|\le T$, where we may take $C_1 = \frac{C(1+T)}{r_0^2}$.
We arrive at \eqref{est_eu} by integrating this expression from $0$ to $t$.

\end{proof}

\section{statement of main  theorem}\label{S:statement}

In this section we state our main theorem, which relates a semilinear
wave equation to a hypersurface of prescribed mean curvature on a
Lorentzian manifold. We first introduce these ingredients.

We will always use greek letters such as  $\alpha, \beta,\ldots$ to denote indices
that run between $0$ and $n$, and we implicitly sum repeated greek indices from $0$ to $n$.
We will explicitly indicate sums that run over different ranges,
we will  generally not use greek letters for indices 
that belong to some proper subset of $\{0,\ldots, n\}$.

\subsection{the Lorentzian manifold}\label{ss:tlm}

We consider a manifold $N$
%
that we assume to be homeomorphic to $\R^{1+n}$, $n\ge 2$, and we
fix global coordinates $(x^0,\ldots, x^n)$. 
We will write $\bf h$ to denote a Lorentzian inner product on $N$, and we let
$(h_{\alpha\beta} )$
denote the components of the metric tensor with respect to the given coordinates.
We also write
\[
(h^{\alpha\beta}) = (h_{\alpha\beta})^{-1}  ,\qquad\qquad h = \det (h_{\alpha\beta}).
\]
We will assume that $(h_{\alpha\beta})$ is smooth and that  there exists some constant $c_1>0$ such that
\beq
h_{00} \le -c_1, \quad\qquad \sum_{i,j=1}^n h_{ij} \xi^i \xi^j \ge c_1 |\xi|^2,
\qquad\quad
| h_{\alpha\beta}| \le  c_1^{-1}\mbox{ for all }\alpha,\beta.
\label{lorentzian}\eeq
everywhere in $N$. Thus $x^0$ may be thought of as a time coordinate, and we
will sometimes write $t$ for $x^0$. We will further assume that
\beq
h_{0i} = 0,\qquad i=1,\ldots, n.
\label{hsplit}\eeq
Then it is clear that $0 < c \le -h \le C$ everywhere in $N$.
For $0 \le t \le  T <\infty$, we will use the notation 
\[
\Sigma_t := \{ x\in N : x^0 = t\},\qquad\quad N_T :=
\{ x\in N : |x^0|< T \}.
\]

In view of \eqref{lorentzian}, we can obtain a Riemanian metric from ${\bf h}$ by changing the
sign of $h_{00}$. Since this metric is uniformly
equivalent to the Euclidean metric, and the associated
volume element  is uniformly comparable
to the Euclidean volume element on $\R^{1+n}$,
we will for simplicity use the
Euclidean structure on $\R^{1+n}$ to define
$L^p$ and Sobolev norms on $N$.

\subsection{ the semilinear wave equation}\label{ss:slw}
Let $\kappa:N\to \R$ be a fixed smooth function. 
We will consider the equation
\beq
\label{nlwN}
\e \Box_{\bf h} u  + \frac 1 \e f_0(u)  +\kappa f_1(u) = 0, \qquad u : N\to \R
\eeq
where 
\[
\Box_{\bf h} u :=  \frac {-1}{\sqrt{-h}} \partial_{x^\alpha}\left( {\sqrt{-h}}h^{\alpha\beta}\partial_{x^\beta} u \right).
\] 
It follows from \eqref{lorentzian} that \eqref{nlwN} is a hyperbolic equation.
We always assume that  the nonlinearities
$f_0,f_1$ in  \eqref{nlwN} have the form
\beq\label{f0f1}
f_0 = F', \qquad
f_1 = \begin{cases}
\sqrt{ 2F} &\mbox{ in }[-1,1]\\
-\sqrt{ 2F} &\mbox{ elsewhere }.
\end{cases}
\eeq
for $F:\R\to \R$ a smooth function such that 
\beq
F(x) > 0 \mbox{ if }|x|\ne 1,\qquad
c(1-|x|)^2 \le F(x) \le C(1-|x|)^2 \mbox{ if }|x|\le 2.
\label{doublewell}\eeq
Note that $f_0, f_1$ are smooth as a consequence of \eqref{doublewell} and the
smoothness of $F$.

We do not address questions about well-posedness of \eqref{nlwN}.
When $(N, {\bf h})$ is flat Minkowski space $\R^{1+n}$ then
global well-posedness in the energy space can be guaranteed
by imposing suitable growth conditions on $F$. In
particular, if $n\le 4$ then \eqref{nlwN} on Minkowski space is
globally well-posed for 
$f_0(u) = 2(u^2-1)u$ and $f_1(u) = 1- u^2$, associated
to the potential $F(u) = \frac 12(1-u^2)^2$.

The form of the nonlinearity is further discussed in Section \ref{ss:nonlin}, 
where we show that the assumptions \eqref{f0f1} are
not actually restrictive if $\kappa$ is constant.
Note also that one can easily generate nonlinearities satisfying the above conditions by
starting from $f_1$ such that $\sign f_1(s) = \sign(1-s^2)$, $|f_1'(\pm 1)| \ne 0$,
then defining $F = \frac 12 f_1^2$ and $f_0 = F' = f_1 f_1'$.

\subsection{ the hypersurface of prescribed mean curvature}\label{ss:hpmc}

We assume that $I$ is a bounded open subset of $N_{T^*}$, for some $T^*>0$, such that $\Gamma := \partial I \cap N_{T^*}$ is a smooth embedded timelike hypersurface satisfying
the prescribed mean curvature\footnote{Our sign conventions for the unit normal, and hence the mean curvature, are described in 
Section \ref{ss:mc}, where we also review some basic properties of mean curvature.
These sign conventions are such that the curve around which the solution in
Section $2$ concentrates, with the orientation we have implicitly chosen there, in fact  has ``mean curvature" equal to $-\kappa$ rather than $\kappa$.}
 equation
\beq
-\kappa(x) =  \mbox{mean curvature  in $(N, {\bf h})$ of $\Gamma$ at $x$}
\label{pmcm}\eeq
and that
\beq
\mbox{$\Gamma$ is orthogonal to the initial hypersurface $\Sigma_0$. }
\label{zerovelocity}\eeq
This means that $\Gamma$ has zero initial velocity with respect to the
initial hypersurface. 
See \eqref{Psi1}, \eqref{zv2} below for  a precise formulation.

For smooth data and smooth $\kappa$, such as we consider here, local existence of smooth embedded submanifolds $\Gamma\subset N$ satisfying
\eqref{pmcm}, \eqref{zerovelocity}
follows from arguments  in Milbredt \cite{mil}.

\begin{remark}
In fact  Milbredt \cite{mil} studies a rather general Cauchy problem for
the $\kappa = 0$ case of \eqref{pmcm} on a larger class of Lorentzian manifolds than
we consider here. Modifying his arguments to extend to the case
of smooth $\kappa$ presents no difficulty. His basic existence results are 
ultimately proved by
solving \eqref{pmcm} in coordinate charts (with a suitable choice of gauge)
and then piecing together these local solutions, using finite propagation 
speed for the equation. Solvability in coordinate charts depends on
local solvability results for the Cauchy problem for
a general quasilinear hyperbolic equation, of the form
\[
g^{\alpha\beta}(\Psi, D\Psi) \partial_{\alpha} \partial_{\beta}\Psi = f(\Psi, D\Psi),
\]
where
$g^{00}\le - \lambda$ and $g^{ij} \ge \mu \delta^{ij}$.
Changing the equation to allow  nonzero $\kappa$
simply adds some additional smooth lower-order terms on the right-hand
side, and  the existence results used in \cite{mil} apply with no change
to the equation once it is modified in this way. Other aspects of the argument,
such as piecing together local solutions, are similarly uneffected.
\label{rem:milbredt}\end{remark}

\subsection{main theorem}

The main result of this paper is the following. The statement uses terminology and notation introduced
in Sections \ref{ss:tlm} - \ref{ss:hpmc}, and
for any set $U$, we will write
\[
\mbox{sign}_U(x) := \begin{cases}
1&\mbox{ if }x\in U,\\
-1&\mbox{ if not}.
\end{cases}
\]

\begin{theorem}\label{T.mink}
Assume that $T_0 < T^*$. 

Then there exists 
a neighborhood $N'$ of $\Gamma \cap N_{T_0}$
and a unique smooth function $d_\Gamma:N'\to \R$ such that
\beq\label{dGamma}
d_\Gamma(x) = 0 \mbox{ on }\Gamma,\qquad
h^{\alpha \beta} \, \partial_{x^\alpha} d_\Gamma \, \partial_{x^\beta} d_\Gamma = 1, \qquad d_\Gamma>0 \mbox{ in }I\cap N'.
\eeq
Moreover, $d_\Gamma$ is bounded away from $0$ outside of every
neighbourhood of $\Gamma$.

In addition, for every $\e\in (0,1]$, there exists  smooth 
initial data $(u_0,u_1)\in \dot H^1\times L^2 (\Sigma_0)$
such that, if $u$ is a smooth solution of \eqref{nlwN} with $(u, \partial_{x^0}u) = (u_0,u_1)$ on $\Sigma_0$,
then
\beq\label{L2est}
\int_{N_{T_0}} |u - \mbox{sign}_I |^2 \ \le \ C \e,
\eeq
and
\beq
\int_{N_{T_0}} \left[ d_\Gamma^2\, \chara_{N'}  +  \chara_{N_{T_0}\setminus N'}\right]
\left(\e |Du|^2 + \frac 1 \e F(u) \right) \ 
\le \  C \e^2.
\label{weightedenergy}\eeq
Here $C$ is a constant that depends on ${\bf h}, F, \Gamma$ but is independent
of $\e$.
\end{theorem}

The function $d_\Gamma$ from the theorem is
the signed distance to $\Gamma$ with respect to the ${\bf h}$ metric.
In fact under our hypotheses it is uniformly comparable to the signed Euclidean distance
to $\Gamma$, so we could replace $d_\Gamma$ in \eqref{weightedenergy}
by the Euclidean squared distance with changes only to constants (depending on 
$\Gamma$ and the choice of the neighborhood $N'$).

Our proof yields additional information that we have not recorded in the statement of the theorem, including the following:
\begin{itemize}
\item We show that $\int \e |Du|^2 + \frac 1 \e F(u) \ge C >0$ for
all small $\e$, so \eqref{weightedenergy} implies that the energy
is strongly concentrated near $\Gamma$.
\item We find  certain vector fields $X$, depending only on the geometry of $\Gamma$,
such that  $\| X \cdot Du \|_{L^2(N_{T_0})}^2 \le C \e$.

\item Our arguments in fact establish not just estimates of some specific
solutions, but also more general stability estimates,
see Proposition  \ref{localest} for example.

\end{itemize}

\subsection{discussion}

As in \cite{j-mms} and Section \ref{S:simple}, 
the heart of the proof of Theorem \ref{T.mink} 
consists of weighted energy estimates in
well-chosen coordinates near $\Gamma$.
In particular, we will introduce coordinates $(y^0,\ldots, \yn)$ such that 
$\Gamma = \{ (y^0,\ldots, \yn) : \yn = 0 \}$, and in addition
$\yn \mapsto c(\yn) = (y^0,\ldots, \yn)$ is 
(approximately) a geodesic with respect to the Lorentzian metric
for every $(y^0,\ldots, y^{n-1})$,
with $c(0)\in \Gamma$ and $c'(0)$ normal to $\Gamma$.
A key point is that the geometry of
$\Gamma$ is exactly such that, when the equation \eqref{nlwN}
is written in the these coordinates, some cancellations occur that
make very strong energy estimates possible.

\smallskip

Writing $v$ to denote the solution of \eqref{nlwN} in the $y$ coordinates, 
the initial data we consider will have the form
\beq
v(0,y^1,\ldots, \yn) \approx q(\frac \yn \e) =: q_\e(\yn),
\qquad\quad
\partial_{y^0}v(0,y^1,\ldots, \yn) = 0
\label{v.data00}\eeq
near $\Gamma = \{ \yn = 0 \}$, where $q$ solves
\beq
-q''+ f_0(q) = 0,\qquad q(0)=0, \qquad q(s)\to \pm 1\mbox{ as }s\to \pm\infty.
\label{q.def}\eeq
Existence of a profile $q$  solving \eqref{q.def} is standard. Indeed, 
multiplying by $q'$, integrating, and using \eqref{f0f1},
one finds that the unique solution of
\beq
q' - f_1(q)=0, \qquad q(0)=0
\label{q.def2}\eeq
also satisfies \eqref{q.def}. 
We remark for future reference that standard ODE arguments, using 
the fact that $f_1'(\pm 1) >0$ (see \eqref{doublewell}), imply that
\beq
|q(s) - \sign(s)| \le C e^{- c|s|} \qquad \mbox{ as $s \to  \pm\infty$}.
\label{decay}\eeq

As is well-known, the profile $q_\e$ is characterized by an optimality property.
Indeed, for any  $\wt q:\R\to \R$,
\[
f_1(\wt q) \wt q' \  \le \ 
\frac \e 2 \wt q'^2 + \frac 1{2\e}f_1^2(\wt q)  \ = \ 
\frac \e 2 \wt q'^2 + \frac 1{\e}F(\wt q) 
\]
by \eqref{f0f1}.
Thus if $\wt q(s)\to \pm 1$ as $s\to \pm\infty$, then
\beq
c_0 := \int_{-1}^1 f_1(s)\, ds\ 
\ =  \ \int_{-\infty}^\infty f_1(\wt q(s)) \, \wt q'(s)\, ds  \   \le
\int_{-\infty}^\infty \frac \e 2 \wt q'^2 + \frac 1 {\e}F(\wt q) \,ds.
\label{q.optimal}\eeq
Moreover, equality holds if and only if $\e\wt q' = f_1( \wt q)$, which 
occurs exactly when $\wt q$ is a translate of
the profile $q_\e$  above.

\smallskip

Formal arguments suggest that the solution $v$ with
initial data \eqref{v.data00} should satisfy
$
v(y^0,\ldots, \yn) \approx q_\e(\yn).
$
As in Remark \ref{r:extra}, our basic estimate \eqref{L2est}
in fact implies that $\int [ v(y^0,\ldots, \yn) - q_\e(\yn)]^2 \le C \e$,
but is not sharp enough
to distinguish the shape of the profile; that is, it does not allow us to say
whether $\sign(\yn)$ or $q_\e(\yn)$ is closer to $v$.
But, again as in Remark \ref{r:extra}, estimates established in the course
of the proof  in fact imply that for most $(y^0,\ldots, y^{n-1})$,
\[
\int \left| v(y^1,\ldots, y^{n-1}, \yn) - q( \frac{( \yn  -  y^n_0 )}\e)  \right|^2  \frac{d\yn}\e \le  C \e^2
\]
for some translation $y^n_0 = y^n_0(y^0,\ldots, y^{n-1})$
such that $|y^n_0| \le C \e$. 
Indeed, this follows from estimates we establish
in Proposition \ref{localest} of $\zeta_1$,
defined in \eqref{zeta1.def} below, together with spectral estimates like
those discussed in Remark \ref{r:extra}.

\subsection{about the  nonlinearity}\label{ss:nonlin}

In equation \eqref{nlwN}, 
we have assumed a nonlinearity $f_\e(x,u) :=  f_0(u) + \e\kappa(x) f_1(u)$,
for $f_0,f_1$ satisfying \eqref{f0f1},
that appears to have a very special form.

This is not actually the case  if $f_\e$ depends only on $u$; in this case a 
nonlinearity $f_\e$ associated to a general 
double-well potential can be written in this form.
Indeed,
assume that $f_\e = F_\e'$,  
where $F_\e:\R\to \R$ is 
a smooth function
with nondegenerate
local minima at two points, say $\pm 1$, a local maximum
at some point in $(-1,1)$, and no other critical points.
We claim that if these hold, then there exists a 
number $\kappa$ and a smooth,
nonnegative function $F$ satisfying
\eqref{doublewell} (and in particular 
vanishing exactly at $x = \pm 1$), such that
\beq
f_\e = F' + \e \kappa \sqrt{2F}\sign_{(-1,1)}.
\label{decompose_f}\eeq
If we define $f_0 = F'$ and $f_1 = \sqrt{2F}\sign_{(-1,1)}$, then this is exactly \eqref{f0f1}.
We sketch a proof:
For $\kappa>0$, 
let $F^\kappa(s)$ be the unique solution for $s >-1$
of the ODE \eqref{decompose_f} with initial data 
$F^\kappa(-1) = 0$.
(Despite the non-Lipschitz nonlinearity, uniqueness can be deduced from the nondegeneracy of $F_\e$,
which implies that $f_\e'(\pm 1)>0$.)
Then a shooting argument, using properties  of $f_\e$ that follow from
our assumptions about $F_\e$, shows that there is exactly one choice of $\kappa$
such that $F^\kappa(s)$ exists and is positive for $s\in (-1,1)$, and in addition
$F^\kappa(1)= 0$. We then define $F^\kappa$ outside of $[-1,1]$
by requiring that it solve \eqref{decompose_f} everywhere. 
Then $F = F^\kappa$ satisfies \eqref{doublewell}
and also solves \eqref{decompose_f} for the
value of $\kappa$ found in this argument. 

\smallskip

On the other hand, if $f_\e$ depends nontrivially on $x$, then
the form of $f_\e$ has very particular useful properties that will be exploited in our analysis.
Notably, for $f_\e$ of this form, the optimal profile $q_\e$
is formally independent of the value of $\kappa$.
The point is that $q_\e$ can be characterized
in terms of either (scaled versions of) \eqref{q.def} or \eqref{q.def2},
and hence satisfies
\beq\label{trwave3}
\e( q'' + \kappa q') = \frac 1 \e f_\e(q) = \frac 1 \e(f_0(q) + \e \kappa f_1(q))  
\eeq
for every $\kappa\in \R$. (Related issues appear also in 
\eqref{trwave1}, \eqref{trwave2}.) This simplifies our analysis.
We believe, however, that our arguments could be adapted to study nonlinearities with more general  dependence on $x$,
at the expense of some technical complications and probably weaker estimates.

\section{an adapted coordinate system}\label{S:adapted}

We now introduce the coordinate
system near $\Gamma = \partial I\cap N_{T^*}$ in which our main estimates will take place.

\subsection{ a good parametrization of $\Gamma$}\label{S:goodp}

Fix  a smooth $(n-1)$-dimensional manifold $M$ diffeomorphic to 
$\Gamma_0 := \Gamma \cap \Sigma_0$. Our assumptions imply that $M$ is compact.
The example that arises most
naturally in cosmological settings is  $M = S^{n-1}$.
We will  write $(y^1,\ldots, y^{n-1})$, or simply $y'$, to denote
local coordinates on $M$.  

For $T>0$, we will write
\beq
M_T := (-T,T)\times M,\qquad
\label{MTrho}\eeq
We will always use
{\em standard local coordinates}
on $M_T$, by which we mean
coordinates of the form
$(y^0,y')$,
where $y^0\in (-T,T)$ and 
$y' = (y^1,\ldots, y^{n-1})$ are local coordinates on $M$.
We will often write $\yt = (y^0,y')$ to denote a point in $M_T$, where the
superscript $\tau$ stands for ``tangential".

We will parametrize $\Gamma\subset N_{T^*}$ by a smooth map
$\Psi: M_{T^*} \to N_{T^*}$
of the form 
\begin{equation}
\Psi(\yt) = (y^0, \psi(\yt))\quad\mbox{ for some }
\psi: M_{T^*}\to \R^n.
\label{Psi1}\end{equation}
(Here and below, we use the fixed coordinate system on $N$ to identify it
with $\R^{1+n}$.)
Note that with this convention and condition \eqref{hsplit} on the metric $(h_{\alpha\beta})$, 
assumption \eqref{zerovelocity}
becomes
\beq\label{zv2}
\frac{\partial \psi}{\partial y^0}(0,y') = 0\qquad\mbox{ for all }y'\in M.
\eeq
 We also impose the condition
\begin{equation}
\gamma_{0a} = \gamma_{a0}  = 0
\qquad\mbox
{for }a=1,\ldots, n-1,
\label{Psi2}\end{equation}
where here and in what follows, we use the notation
\begin{equation}
\gamma_{ab}  : = {\bf h}(\frac{\partial\Psi}{\partial y^a}, 
\frac{\partial\Psi}{\partial y^b}) \qquad\qquad\mbox
{for }a,b =0,\ldots, n-1.
\label{gammaab}\end{equation}
The existence of such a parametrization  $\Psi$
is rather standard. Indeed, suppose we
are given local coordinates
$y' = (y^1,\ldots, y^{n-1})$ on a subset of $M$. By definition
$M$ is diffeomorphic to $\Gamma_0$, so we may fix a diffeomorphism
$y'\to \psi_0(y')\in \Gamma_0$.
Then for every $y'$ and sufficiently small $\delta>0$
there exists a unique curve
$p = p(\cdot; y'):(-\delta, \delta)\to N$, with components $(p^0,\ldots, p^n)$,
such that 
\[
p^0(t) = t,
\qquad
p(0; y') = \psi_0(y'),
\qquad
{\bf h}(p'(t) , \frac{\partial\Psi}{\partial y^a})= 0
\quad\mbox{for  $a =0,\ldots, n-1$ and $|t| < \delta$.} 
\]
Indeed, in coordinates  this is just a first-order ODE for $p(\cdot)$ to which standard theorems
apply.
It is then easy to see that
$\delta = \delta(y')$ is bounded away from zero on small enough subsets,
and on those sets can define  $\Psi(y^0, y') := p(y^0, y')$.

\subsection{almost-normal coordinates near the hypersurface}
For $r$ positive, we will write 
\beq\label{Mrho}
M^r :=  M\times (-r,r).
\qquad\quad
M^r_T :=M_T \times (-r,r).
\eeq
As above, we will always use standard local coordinates
on these spaces, that is, coordinate systems that respect
the product structure. Thus, in these coordinates,
points in $M_T^r$ have the form $(\yt, \yn) = (y^0, y', \yn)$, 
where $y'$ are local coordinates on $M$,
$|y^0|<T$ and $|\yn|<r$.

The next proposition introduces a map $\phi:M_T^{2\rho}\to N$
that parametrizes a neighbourhood of $\Gamma$ and such that
certain good properties are enjoyed by the pullback metric
\[
g_{\alpha\beta} := {\bf h}( 
\frac {\partial \phi}{\partial y^\alpha},
\frac {\partial \phi}{\partial y^\beta}),
\qquad
(g^{\alpha\beta}) := (g_{\alpha\beta})^{-1}, \qquad 
g:= \det (g_{\alpha\beta})
\]
where  $\alpha, \beta \in \{0,\ldots, n\}$.
The proof will be deferred to Section  \ref{S:Phi}. We remark however
that $\phi$ essentially defines a Gaussian normal coordinate system
near $\Gamma$, modified slightly to arrange that condition
\eqref{Phi_b} below holds; this condition implies that
changing variables using $\phi$ maps Cauchy problems 
for \eqref{nlwN}, with data given for $x^0=0$,  to
Cauchy problems with data given on the hypersurface $\{ y^0=0 \}$.
This will be useful.

\begin{proposition}\label{Prop_Phi}
For every $T_0<T^*$, there exists 
$\phi:M_T^{2\rho} \to N$, for some
$T\in (T_0, T^*)$ and $\rho>0$, such that
 such that 
$\phi$ is a diffeomorphism 
onto its image, and 
the following hold. First, $\phi(\yt, 0) = \Psi(\yt)$, which implies that
\beq
\Gamma \cap N_{T_0}  = \phi ( M_{T_0} \times \{0\}),
\label{Phi_a}\eeq
and hence that
\beq
N' := \phi( M_T^{2\rho}) \mbox{ is an open neighbourhood of $\Gamma\cap N_{T_0}$.}
\label{Nprime.def}\eeq
In addition,
\beq
\mbox{ if $y^0\in \bar M_T^{2\rho} $ and $|y^0| = T$, then $|x^0|> T_0$ for $x = \phi(y)$.}
\label{Trho}\eeq
Second, 
\beq\label{Phi_b}
\phi( \{0\} \times M^{2\rho}) \subset \Sigma_0 = \{ x\in N\cong \R^{1+n} : x^0=0\}.
\eeq
Third, the metric satisfies
\beq
(g_{\alpha\beta})(\yt, \yn) = 
\left(
\begin{array}{ll}
(\gamma_{ab})(\yt) &0\\
0& \mbox{1}
\end{array}\right) +
O
\left(
\begin{array}{ll}
|\yn|&|\yn|  \\
|\yn|&(\yn)^2 
\end{array}\right) 
\label{glowerab}\eeq
(in block  form), where $(\gamma_{ab})$ was introduced in
\eqref{gammaab}. 
Hence
\beq
(g^{\alpha\beta})(\yt, \yn) = 
\left(
\begin{array}{ll}
(\gamma^{ab})(\yt) &0\\
0& \mbox{1}
\end{array}\right) +
O
\left(
\begin{array}{ll}
|\yn|&|\yn|  \\
|\yn|&(\yn)^2 
\end{array}\right) 
\label{gupperab}\eeq
In addition,
\beq
(\partial_0 g^{\alpha\beta})(\yt, \yn) = 
O
\left(
\begin{array}{ll}
1&|\yn|  \\
|\yn|&(\yn)^2 
\end{array}\right) 
\label{dtgupperab}\eeq
Next, the eikonal equation \eqref{dGamma}
has a unique smooth solution $d_\Gamma$ on $N'$,
and if we define $\pi^n(y^0,\ldots, \yn) = \yn$, then
\beq
\pi^n\circ \phi^{-1} = d_\Gamma +O(d_\Gamma^2).
\label{approx.d}\eeq
Finally,
\beq
\frac {-1}{\sqrt{-g}} g^{n\alpha}\partial_{y^\alpha} \sqrt{-g}  \ = \ 
\frac {-1}{\sqrt{-g}} \partial_{\yn} \sqrt{-g} +O(|\yn|)   \ = \ 
-\kappa(\yt) + O(|\yn|). 
\label{g3}\eeq
\end{proposition}

\begin{remark}\label{rem:uniform}
The implied constants in the above estimates could in principle
depend on the choice of local coordinates for $M$. However,
our assumptions imply that $M$ is compact, and so we may  once and for all fix
a cover of $M$ by a finite collection
of coordinate neighborhoods $U_1,\ldots, U_k$, such that
all the above estimates are uniform on $(-T,T)\times U_i\times (-2\rho, 2\rho)$
for every $i$. (This last point will be evident from
the proof of the proposition.)
We can then require that all subsequent computations in local coordinates
are carried out in one of these fixed coordinate systems. Having done this,
all the above constants are uniform on $M_T^{2\rho}$.
The same remark
applies below as well.
\end{remark}

Our later energy estimates will contain a  symmetric tensor $(a^{\alpha\beta})$, defined by
\beq
\frac 12 a^{\alpha\beta}  \xi_\alpha \xi_\beta  := -g^{0\alpha} \xi_\alpha \xi_0 + \frac 12 g^{\alpha\beta}\xi_\alpha\xi_\beta
=
- \frac 12 g^{00} \xi_0^2 +  \frac 12 \sum_{i,j=1}^n g^{ij} \xi_i \xi_j.
\label{a.def}\eeq
It follows from \eqref{gupperab} that if $\rho$ is taken to be small enough
(which we henceforth assume to be the case) then
there exists some positive constants $c_2, c_3, c_4, c_5$ such that
\beq
\frac 12 \sum_{a,b=0}^{n-1}a^{ab} \xi_a\, \xi_b \ + (1+ (\yn)^2)\  \xi_n^2
\ \le \ 
(1+ c_2(\yn)^2) a^{\alpha\beta}\xi_\alpha \xi_\beta \ \le \ 
2 \sum_{a,b=0}^{n-1}a^{ab}\xi_a \xi_a \ + (1+ c_3(\yn)^2) \xi_n^2,
\label{apositive}\eeq
\beq
\sum_{a,b=0}^{n-1}\delta^{ab}\xi_a\xi_b \le c_4 \sum_{a,b=0}^{n-1}a^{ab}\xi_a\xi_b,
\label{a++}\eeq
and
\beq
|g^{n\alpha}\xi_\alpha \xi_0| \ \le \  \frac {c_5}2 a^{\alpha\beta}\xi_\alpha \xi_\beta
\label{a+++}\eeq
everywhere in $M^{2\rho}_T$, 
for all $\xi\in \R^{1+n}$.

\section{weighted energy estimates in normal coordinates}\label{S:weenc}

In this section we study the wave equation
\beq
\e \Box_{\bf g} v + \frac 1\e f_0(v) + \kappa f_1(v ) = 0, \qquad
v:M_T^{2\rho}\to \R
\label{nlwM}\eeq
Here ${\bf g}$ is a metric on $M_T^{2\rho}$ satisfying the
conclusions of Proposition \ref{Prop_Phi}, $\kappa$ is a smooth, bounded function on
$M_T^{2\rho}$, and $f_0,f_1$ satisfy \eqref{f0f1}. 
In particular, if $u$ solves \eqref{nlwN} on $N$ then $v := u \circ \phi$
solves \eqref{nlwM} on $M_T^{2\rho}$.

To state our main estimates, we need some notation.
We start by fixing a smooth volume form $d(vol)_0$ on $M$, and
we extend it to a volume form $d(vol)$ on $M_T^{2\rho}$ by requiring that 
\beq
d(vol) = dy^0 \wedge d(vol)_0 \wedge d\yn
\label{extendvol}\eeq
in standard local coordinates. 
We emphasize that $d(vol)$ in general does {\em not} coincide with the volume form
associated to the Lorentzian metric ${\bf g}$.

 We will similarly extend $d(vol)_0$
to $M_T$ and $M^{2\rho}$, writing $d(vol)$ in every case; the meaning
should always be clear from the context.

Thus, in standard local coordinates these are represented by expressions
of the form
\begin{align*}
d(vol)_0 &=\  \omega_0(y') \, dy^1\wedge\cdots \wedge dy^n,\\
d(vol) &=  \ \omega(y) \, dy^0 \wedge dy^1\wedge\cdots \wedge dy^n \wedge d\yn \qquad\mbox{ on }M_T^{2\rho}
\end{align*}
where $\omega(y^0,y', \yn) = \omega_0(y')$ in $M_T^{2\rho}$. Here $\omega_0$ is a smooth positive function that depends on our choice of local coordinates for $M$.
Similar expressions hold for $d(vol)$ on $M_T$ and $M^{2\rho}$.

Next, we define a natural energy density associated to \eqref{nlwM}. For  $v\in H^1(M_T^{2\rho})$,
let
\beq
e_\e(v; {\bf g}) := \frac \e 2\, a^{\alpha\beta}  \partial_{y^\alpha} v \  \partial_{y^\beta}v
+ \frac 1{\e} F(v),
\label{eeps.def}\eeq
for $F$ defined in \eqref{doublewell},  and $a^{\alpha\beta}$ defined in \eqref{a.def}.
We will write simply $e_\e(v)$ when there is no ambiguity, which will be
the case throughout this section. Finally, recall that we have defined
\[
c_0 := \int_{-1}^{1} f_1(u) \ du,
\]
and that, as noted in \eqref{q.optimal},
$c_0$ is a lower bound for the energy of a $1$-d interface connecting
the equilibrium states $\{\pm 1\}$, and this lower bound is attained by the profile $q$.

The following estimates are the heart of the proof of Theorem \ref{T.mink}.


\begin{proposition}
Let $v$ be a smooth solution of \eqref{nlwM} on $M_T^{2\rho}$, and 
assume that ${\bf g}$ satisfies the conclusions of Proposition \ref{Prop_Phi}.
Define $\rho(s) = \rho - c_5s$, for $c_5$ defined in \eqref{a+++}, and
\begin{align}
\zeta_1(s) 
&:= 
\ 
\left. \int_{M^{\rho(s)}} (1+c_2(\yn)^2)\, e_\e(v) d(vol) \right|_{y^0=s} - c_0 vol_0(M) \ 
 \label{zeta1.def}\\
\zeta_2(s) 
&:= 
\left.\int_{M^{\rho/2}} |\yn| \ |v - \sign (\yn)|^2 d(vol)\right|_{y^0=s} 
\label{zeta2.def} \\
\zeta_3(s) 
&:= 
\left. \int_{  M^{\rho(s)}}
\frac \e 2\sum_{a,b = 0}^{n-1} a^{ab}\, v_{y^a}\, v_{y^b} +  (\yn)^2 \left[  \frac \e 2  |\partial_{\yn} v|^2 + \frac 1{\e}F(v) \right] \ d(vol) 
\right|_{y^0=s}
\label{zeta3.def}\end{align}
Then there exists a constant $C$, independent of $v$ and of $\e\in (0,1]$, such that
\beq\label{Ploc.c1}
\zeta_i(s) \le C  
\max \big(  \zeta_1(0), \zeta_2(0) \big)
\qquad\mbox{ for }i=1,2,3\mbox{ and } 0<s < s_1 := \min( T, \rho/(3 c_5))
\eeq
\label{localest}\end{proposition}

Note that the constants in Proposition \ref{localest} may depend 
for example on 
$vol_0(M), \| \kappa\|_{\infty}$, constants in Proposition \ref{Prop_Phi}
(which may depend on $T_0$), but they are independent of $\e\in (0,1]$.


\subsection{differential energy inequality}

\begin{lemma}
Assume the hypotheses of Proposition \ref{localest}.
Then 
\beq
\frac \partial {\partial y^0} e_\e(v) 
\le 
\e C ( \sum_{a,b=0}^{n-1}a^{\alpha\beta} \partial_{y^\alpha}v\ \partial_{y^\beta}v
%
%
+ |\yn|^2 \ |\partial_\yn v|^2)
+ 
\e \divf_{M^{2\rho}} \vp 
+ 
\kappa\big[ \e v_{\yn}  - f_1(v)\big] \cdot v_{y^0}
\label{eep.prime1}\eeq
 where
\begin{equation}
\vp := (\vp^1,\ldots, \vp^n),
\quad\quad
\quad
\vp^i := 
\ g^{i\alpha }  v_{y^\alpha} \cdot v_{y^0}
\label{vp.def}\end{equation}
and $\divf_{M^{2\rho}}$ denotes the
divergence on $M^{2\rho}$ with respect to the fixed volume form $d(vol)$,
so that 
\beq
\divf_{M^{2\rho}}\vp := \frac 1{\omega} \sum_{i=1}^n \partial_{y^i} (  \omega \, \vp^i
).
\label{divM.def}\eeq
in standard local coordinates on $M^{2\rho}$.
\label{L.eflux}\end{lemma}

Note that if we compare \eqref{eep.prime1} to \eqref{polar.energy}, then the term here
corresponding to ``Term 1" in \eqref{polar.energy} has been absorbed into the
first term on the right-hand side of \eqref{eep.prime1}

\begin{proof}In standard local coordinates on $M_T^{2\rho}$, our equation \eqref{nlwM} takes the form
\[
 \frac {-\e}{\sqrt{-g}} \partial_{y^\alpha}\left( {\sqrt{-g}}g^{\alpha\beta}\partial_{y^\beta} v \right)
  + \frac 1 \e f_0(v)  + \kappa f_1(v) = 0.
\]
We rewrite the leading term as a divergence with respect to $d(vol)$ on $M_T^{2\rho}$, 
leading to
\[
-\frac {\e}{\omega} \partial_{y^\alpha}\left( \omega g^{\alpha\beta}\partial_{y^\beta} v \right)
- \e b^\alpha \partial_{y^\alpha} v
+ \frac 1 \e f_0(v) + \kappa f_1(v) = 0
\]
where
\beq
b^\beta :=
\frac {\omega} {\sqrt{-g}}  \ g^{\alpha\beta} \  \partial_{y^\alpha}\left( \frac{ \sqrt{-g}}{\omega} \right) .
\label{b.def}\eeq
%
%
Multiply this by $v_{y^0}$ and rewrite, recalling \eqref{f0f1}, to find that
\[ 
- \frac\e\omega\partial_{y^\alpha}( \omega g^{\alpha\beta} v_{y^\beta} \ v_{y^0})
+\e g^{\alpha\beta} v_{y^\beta}  \, v_{y^0 y^\alpha}
+\frac 1{\e}F(v)_{y^0}  
= 
 \e  b^\alpha v_{y^\alpha}   v_{y^0} -  \kappa f_1(v) \, v_{y^0}
\] 
Note that
\[
g^{\alpha\beta}v_{y^\beta}  v_{y^0y^\alpha} = \frac12 \partial_{y^0} (g^{\alpha\beta} v_{y^\beta} v_{y^\alpha}) - \frac12 g^{\alpha\beta}_{y^0} v_{y^\beta}  v_{y^\alpha}
\]
and that $\partial_{y^0} \omega = \partial_{y^n}\omega = 0$ on $M_T^{2\rho}$.
We use these facts and collect all the terms of the form $\partial_{y^0} [ \cdots]$ to the left-hand side to obtain
\begin{multline} \label{weight1}
\partial_{y^0} \left[ - \e g^{0\beta} v_{y^\beta} \cdot v_{y^0} + \frac\e2 g^{\alpha\beta} v_{y^\beta} \, v_{y^\alpha} + \frac1\e F(v) \right] \\
= 
\frac \e\omega \partial_i ( \omega g^{i\beta}\partial_{y^\beta}v)
+ \frac\e2 g^{\alpha\beta}_{y^0} v_{y^\alpha} v_{y^\beta}
+ \e  b^\alpha v_{y^\alpha}  v_{y^0}
-  \kappa f_1(v) \, v_{y^0}.
\end{multline}
By definition, the left-hand side is just $\partial_{y^0} e_\e(v)$, and the first
term on the right-hand side is exactly $ \e \divf_{M^{2\rho}} \vp$.
It follows from \eqref{dtgupperab} and \eqref{a++} that
\[
g^{\alpha\beta}_{y^0} v_{y^\alpha} v_{y^\beta}
\le C \sum_{a,b = 0}^{n-1} a^{ab}v_{y^a} v_{y^b}.
\]
To estimate the remaining terms on the right-hand side of \eqref{weight1}
first note that
\begin{align*}
\e b^\alpha v_{y^\alpha} v_{y^0} - \kappa f_1(v) v_{y_0}
&=
\e \sum_{a=0}^{n-1}b^a v_{y^a} v_{y^0} \ + 
\e (b^n - \kappa) v_{y^n} v_{y^0}
+ \kappa [\e  v_{y_n} - f_1(v)] v_{y^0}.
\end{align*}
Also, since $\partial_{\yn}\omega = 0$, 
we see that \eqref{g3} states exactly that $|b^n - \kappa| = O(|\yn|)$.
Thus \eqref{a++} implies that
\[
| (b^n-\kappa) v_{y^n} v_{y^0}| \le  C |\yn  v_{\yn} v_{y^0}| \le C \sum_{a,b = 0}^{n-1} a^{ab}v_{y^a} v_{y^b} + (\yn)^2|v_{\yn}|^2
\]
and that
\[
\sum_{a=0}^{n-1}b^a v_{y^a} v_{y^0} \  \le 
C \sum_{a,b = 0}^{n-1} a^{ab}v_{y^a} v_{y^b} .
\]
Thus the Lemma follows by combining these facts with \eqref{weight1}.\end{proof}

\subsection{stability of the profile}

In this section we collect a couple of lemmas that encode some
stability properties of initial data for which $\zeta_i(0)$ is small, $i=1,2$.
These concern functions of a single variable, which we will 
denote $y^n$, since later we will apply these results to 
the functions $\yn \mapsto v(\yt, \yn)$ for $\yt\in M_T$ fixed.

%
\begin{lemma}
There exists a constant $c_6=c_6(\rho)$ such that if $v \in H^1(-\rho,\rho)$ and if 
\beq
\int_{-\rho}^\rho |\yn| \ | v - \sign(\yn)|^2 d \yn 
\le c_6
\label{L.h1}\eeq
then
\beq
\int_{-\rho}^\rho e_{\e,\nu}(v) \, d\yn \ge  c_0-Ce^{-c/\e},
\qquad\quad
\mbox{ for }
e_{\e,\nu}(v) := \frac\e2 |\partial_{\yn} v|^2 + \frac1\e F(v).
\label{L.lbe}\eeq
Moreover, there exists a constant $c_7 = c_7(\rho)>0$ such that if \eqref{L.h1} holds and if
\beq
\int_{-\rho}^\rho e_{\e,\nu}(v) \, d\yn \ - \ c_0 \  \leq c_7
\label{equipart.hyp}\eeq
then for every $\bar \rho \ge \rho$, 
\beq
\int_{-\bar\rho}^{\bar \rho} \frac 12  \left( \sqrt{\e}  v_{\yn}  - \frac {f_1(v)} {\sqrt\e} \right)^2 \ d\yn
\le
C
\left(\int_{-\bar \rho}^{\bar\rho} e_{\e,\nu}(v)  d\yn - c_0  \right) + C e^{-c/\e}
\label{equipart0}\eeq
as long as $v$ is defined on  $(-\bar \rho, \bar\rho)$, with $C$ independent of $\bar \rho$.
\label{L.1}
\end{lemma}

These are largely proved in Lemma  of \cite{j-mms}, but
since we have modified the statement here in some ways, we 
present some details.

\begin{proof}
Steps 3 and 4 of the proof of  Lemma 11 of  \cite{j-mms} show that if $c_6$ and $c_7$ are chosen to be suitably small and  \eqref{L.h1}, \eqref{equipart.hyp} hold, 
then there exists a function $v_1$ and points
$s_-<s_+$ in $(-\rho, \rho)$ such that
\beq
|v_1(s_\pm) - \pm 1| \le C e^{-c/\e},
\quad
\quad
\quad
\quad
v_1(s) = v(s) \quad\mbox{ if }|s| \ge \rho,
\label{v2.prop}\eeq
and 
\beq
\int_{-\rho}^\rho e_{\e,\nu}(v_1) \le 
\int_{-\rho}^\rho e_{\e,\nu}(v).
\label{v1.prop}\eeq
In fact, $v_1$ is found by minimizing $w\mapsto \int e_{\e,\nu}(w)$ 
among the space of functions
that agree with $v$ outside of the intervals
$(-\rho, -\frac 34 \rho) \cup (\frac 34 \rho, \rho)$. 
Then \eqref{v1.prop} is clear, and
a maximum
principle argument, together with \eqref{L.h1}, \eqref{equipart.hyp}
and the choices of $c_6,c_7$, can be used to 
show that $\pm v_1(\pm \frac 78 \rho) \ge 1 - C e^{-c/\e}$.
If $v_1(\frac 78 \rho)\le 1$ we can take $s_+ = \frac 78 \rho$,
and otherwise we can find some $s_+$ near $\frac 78\rho$ where $v_1(s_+)=1$.
The choice of $s_-$ is similar.

Let $Q(s) := \int_0^s  f_1(t) \ dt$, and note from the definition \eqref{f0f1} of $f_1$ that
\beq\label{modmort}
e_{\e,\nu}(w) \ge \sqrt{2 F(w)} |\partial_{\yn}w|   = |f_1(w) \partial_{\yn}w| = 
|\partial_{\yn} Q(w(\yn))| 
\ge 
\partial_{\yn} Q(w(\yn))
\eeq
for every $w$.
Thus
\beq
\int_a^b e_{\e,\nu}(w)d\yn  \ \ge \  |Q(w(b)) - Q(w(b))|
\label{modmort2}\eeq
for every $w\in H^1$ and every $a<b$.
Applying this inequality with $w=v_1$ and $a = s_-, b=s_+$ rather easily yields
\eqref{L.lbe}; see \cite{j-mms} for a little more detail.

To prove \eqref{equipart0}, assume that $v\in H^1( (-\bar \rho, \bar \rho))$ for $\bar \rho \ge \rho$,
and  note that since $F = \frac 12 f_1^2$,
\[
\int_{-\bar\rho}^{\bar \rho} \frac 12  \left( \sqrt{\e}  v_{\yn}  - \frac {f_1(v)} {\sqrt\e} \right)^2 \ d\yn
=
\left(\int_{-\bar \rho}^{\bar \rho} e_{\e,\nu}(v)  d\yn  - c_0\right)  + \left( c_0 -  Q(v(\bar\rho)) + Q(v(-\bar\rho)) \right)\nonumber \\
\]
Since $v = v_1$ at $\pm \bar \rho$ and $c_0 = Q(1)-Q(-1) = Q(v(s_+))-Q(v(s_-)) + O(e^{-c/\e})$,
we again use \eqref{modmort2} to find that
\begin{align*}
c_0  - Q(v(\bar\rho)) + Q(v(-\bar\rho))
&\le
[Q(v_1(s_+)) - Q(v_1(\bar \rho)) ]
+
[Q(v_1(-\bar\rho)) - Q(v_1(s_-))] + C e^{-c/\e}\\ 
&\le
\int_{-\bar \rho}^{s_-}e_{\e, \nu}(v_1) + \int_{s_+}^{\bar \rho} e_{\e, \nu}(v_1)  + C e^{-c/\e}\\
&=
\left(\int_{-\bar \rho}^{s_-}e_{\e, \nu}(v_1) + \int_{s_+}^{\bar \rho} e_{\e, \nu}(v_1) + c_0 \right) - c_0 + C e^{-c/\e }.\\
\end{align*}


And again using the choice of $s_\pm$ and \eqref{modmort2},
we have
\[
c_0 \  \le \  Q(v_1(s_+) ) - Q(v_1(s_-)) + C e^{-c/\e}
\ \le \  \int_{s_-}^{s_+} e_{\e,\nu}(v_1) + C e^{-c/\e}.
\]
The proof of \eqref{equipart0} is completed by
combining the last three estimates and recalling \eqref{v1.prop}.

\end{proof}

Our next result is exactly Lemma 12 in \cite{j-mms}. Here, in view of future applications, 
we write $v$ as a function of two variables, $y^0\in(0,\tau)$ and $\yn\in (-\rho,\rho)$. 
\begin{lemma}
Let $v\in H^1((0,\tau)\times(-\rho,\rho))$ for some $\tau>0$.
Then there exists a constant $C$, depending on $\rho$ but independent of $\tau$ and of $\e\in(0,1]$, such that
$$
\int_{(-\rho,\rho)} |\yn|\ |v(0,\yn)-v(\tau,\yn)|^2 \, d\yn 
	\le C \int_{(0,\tau)\times(-\rho,\rho)} \frac\e2 v_{y^0}^2 + \frac{(\yn)^2}{\e} F(v) \, d\yn \, dy^0.
$$
\label{L.2}\end{lemma}

\subsection{weighted energy estimate}

Now we give the 
\begin{proof}[Proof of Proposition \ref{localest}]
We will write $\zeta_0 := \max\big( \zeta_1(0), \zeta_2(0) \big)$.

{\bf Step 1. } 
Since
\[
|v(s,y,\yn) -\sign(\yn)|^2 \le 2\left(
|v(s,y,\yn) - v(0,y, \yn)|^2 + |v(0,y, \yn) -\sign(\yn)|^2\right),
\]
we find from Lemma \ref{L.2} that
\begin{align}
\zeta_2(s) 
	& \leq 2 \zeta_2(0) + 2 \int_{M}\int_{-\rho/2}^{\rho/2} \ 
	|\yn| \big| v(s,y', \yn)-v(0,y',\yn) \big|^2 \,  d\yn \ d(vol)_0 \nonumber \\
	& \leq 2 \zeta_0 + C \int_{M} \left( \int_0^s \int_{-\rho/2}^{\rho/2}\  \frac\e2 v_{y^0}^2 + \frac{(\yn)^2}{\e} F(v) \, d\yn\,dy^0 \right) \, d(vol)_0 \nonumber \\
	& \leq 2 \zeta_0 + C \,  \int_0^s \zeta_3(\sigma) \, d\sigma.
\label{zeta2.main}\end{align}

{\bf Step 2.}
For the next few steps of the proof, we regard $s$ as fixed, and we write $v(\cdot)$ rather than
$v(s,\cdot)$.

We will say that  a point $y' \in M$  is {\sl good} if 
\beq
\int_{-\rho/2}^{\rho/2} \ |\yn| \ |v(y', \yn) - \sign(\yn)|^2 d\yn 
 \leq c_6(\rho/2)
\label{goodpt}\eeq
and in addition
\beq
\int_{-\rho(s)}^{\rho(s)} e_{\e,\nu}(v)(y',\yn) d\yn - c_0 \le c_7(\rho/2).
\label{verygoodpt}\eeq
where $c_6,c_7$ were fixed in Lemma \ref{L.1}.
We will say that a point is {\sl bad}  if it is not {\sl good}.

We claim that 
\beq
vol_0 \big ( \{ y' \in M: y' \text{ is {\sl bad}} \} \big)
\le  C\left( \zeta_1(s) + \zeta_2(s) \right)  + C e^{-c/\e}.
\label{badpts.est}\eeq
To prove this, first note that by Chebyshev's inequality,
\begin{align}
vol_0 ( \{ y' \in M:  \mbox{  \eqref{goodpt} fails}\} \big)
	&\leq \frac{1}{c_6}
	\int_M \int_{-\rho/2}^{\rho/2} \ |\yn| \ |v(y', \yn) - \sign(\yn)|^2 d\yn 
	 \, d(vol)_0 \nonumber \\
&	= C \zeta_2(s).
\label{smallbad}
\end{align}
Next, note that
\[
\int_{-\rho(s)}^{\rho(s)} e_{\e,\nu}(v)(s,y',\yn) d\yn - c_0 \ge
\begin{cases}
-C e^{-c/\e} &\mbox{if $y'$ is {\sl good}}\\
c_7 &\mbox{ if $y'$ is {\sl bad} and \eqref{goodpt} holds} \\ 
-c_0&\mbox{ always, and in particular if \eqref{goodpt} fails at $y'$,}\\
\end{cases}
\]
where we have used Lemma \ref{L.1} for the first case.  We integrate to find
\begin{align}
\zeta_1(s) &\ge
\int_M \left(\int_{-\rho(s)}^{\rho(s)} e_{\e,\nu}(v)(s,y',\yn) d\yn - c_0 
\right) d(vol)_0\nonumber\\
&\ge
- c_0 \ vol_0 \big( \{ y' \in M:  \mbox{  \eqref{goodpt} fails}\} \big)
\label{good.a}\\
&\qquad\qquad
+ c_7 \ vol_0 \big( \{ y' \in M:  \mbox{ $y'$ is {\sl bad} but \eqref{goodpt}  holds}\} \big)
 - C e^{-c/\e}. \nonumber
\end{align}
Using \eqref{smallbad}, we deduce that
\[
(vol)_0 \big( \{ y' \in M:  \mbox{ $y'$ is {\sl bad} but \eqref{goodpt}  holds}\} \big)
\ \le C\left( \zeta_1(s) + \zeta_2(s) \right)  + C e^{-c/\e},
\]
and this together with \eqref{smallbad} implies \eqref{badpts.est}.

{\bf Step 3. } Next we estimate $\zeta_3(s)$. We claim that
\beq
\zeta_3(s) \leq \zeta_1(s) +C \zeta_2(s) +C e^{-c/\e}
\label{zeta3.main}\eeq
for every $s$. 
The choice \eqref{apositive} of $c_2$
implies that 
$$
\big(1+c_2(\yn)^2\big) e_\e(v) \geq \frac \e 4 \sum_{a,b=0}^{n-1} a^{ab} \partial_{y^a}v\, \partial_{y^b}v + \big(1+ (\yn)^2\big) e_{\e,\nu}(v).
$$
By the definitions of $\zeta_1$ and $\zeta_3$, it follows that
$$
\zeta_1(s) \geq  \frac 12\zeta_3(s) +  \int_{M^{\rho(s)}} e_{\e,\nu}(v) \, d(vol) - c_0\  vol_0(M).
$$
So to complete the proof of \eqref{zeta3.main}, it suffices to show that
\beq
c_0\  vol_0(M) -  \int_{M^{\rho(s)}} e_{\e,\nu}(v) \, d(vol)
	\leq C \zeta_2(s) + C e^{-c/\e},
\label{step2.mpmp}\eeq
and this follows directly from \eqref{good.a} and \eqref{smallbad}.

{\bf Step 4. } We now claim that
\beq
\zeta_1'(s) \leq C (\zeta_1(s) +  \zeta_2(s)+  \zeta_3(s) )+ Ce^{-c/\e}.
\label{zeta1.main}\eeq%
Recalling the definition $\rho(s) := \rho - c_5 s$, we compute $\zeta_1'(s) = I_1 -  c_5 I_2$, where
\begin{align*}
I_1 & := \int_{\{s\}\times M^{\rho(s)}} \big( 1+(\yn)^2\big) \partial_{y^0} e_\e (v) \, d(vol)  \\
I_2 & := \int_{\{s\}\times M} \big( 1+(\yn)^2\big) e_\e(v) \, d(vol)_0 \bigg|_{\yn=-\rho(s)}^{\yn= \rho(s)}
\end{align*}
It follows directly from the differential energy inequality of Lemma \ref{L.eflux} that 
\begin{align}
I_1 
&\le
C \zeta_3(s) + 
I_{1a} + I_{1b}\label{I1split}
\end{align}
where
\begin{align*}
I_{1a} &:= \int_{\{s\}\times M^{\rho(s)}} \e (1+c_2(\yn)^2) \divf_{M^\rho}\vp \ d(vol),\\
I_{1b} &:= 
\int_{\{s\}\times M^{\rho(s)}} \e (1+c_2(\yn)^2) \ \kappa \,\left( v_{\yn} - \frac1\e f_1(v)\right) \ d(vol).
\end{align*}
{\bf Step 5}. To bound $I_{1a}$, 
note that
\[
 \big( 1+(\yn)^2\big) \divf_{M^{2\rho}} \vp  =  \divf_{M^{2\rho}}  \big(\big( 1+(\yn)^2\big)\vp \big)
 - 2\yn \vp^n
\]
by \eqref{divM.def}, since $\omega$ is independent of $\yn$. 
It is easy to see 
from the definition \eqref{vp.def} of $\vp$ that 
\[
 |\yn| \ |\vp^n|  \overset{\eqref{vp.def}}{ = }   |\yn| \ |g^{n\alpha} \partial_{y^\alpha} v\ \partial _{y^0}v| \le C  \sum_{a,b=0}^{n-1} a^{ab} \partial_{y^a}v\, \partial_{y^b}v +  \e^{-1}|\yn|^2e_{\e, \nu}(v)
\]
and it follows that
\[
\e \int_{\{s\} \times M^{\rho(s)} } |\yn| \ |\vp^n| d(vol) \ \le C \zeta_3(s).
\]
For the other term, 
note that
$\partial (\{s\}\times M^{\rho(s)}) = \{s\} \times M \times \{ -\rho(s),  \rho(s) \}$ (appropriately oriented), and that the induced volume form on $\partial M$
is just $(vol)_0$.
This yields
\[ 
\e \int_{\{s\}\times M^{\rho(s)}}  \divf_{M^{2\rho}}  \big( 1+(\yn)^2 \vp \big)\, d(vol)
	 = \e \int_{\{s\}\times M} \big( 1+(\yn)^2\big) |\vp^n| \, d(vol)_0 \bigg|_{\yn=-\rho(s)}^{\yn=\rho(s)}
\]
Next, the choice \eqref{a+++} of $c_5$ 
was arranged exactly so that
$\e |\vp^n|  \le c_5 e_\e(v)$, so it follows from the above that
\beq\label{I.est1}
I_{1a} - c_5 I_2 \le C \zeta_3(s).
\eeq

\noindent
{\bf Step 6}.
We estimate $I_{1b}$ as follows.
First,
\begin{multline*}
I_{1b} 
	\ \leq \  \frac {\e  \|\kappa\|_\infty}2 \int_{\{s\}\times  M^{\rho(s)}} 
	\left[ |v_{y^0}|^2 +  \left( v_{\yn} - \frac1\e f_1(v) \right)^2 \right] \, d(vol) \\
	\ \leq \ C \zeta_3(s) + 
	C \int_{\{s\}\times  M^{\rho(s)}}   \left(\sqrt{\e} v_{\yn} - \frac {  f_1(v)}{\sqrt \e}\ \right)^2
	  \, d(vol) .
\end{multline*}
Note that Lemma \ref{L.1}  implies that if $y'$ is  {\sl good} in the sense of
\eqref{goodpt}, \eqref{verygoodpt}, then
\[
\int_{-\rho(s)}^{\rho(s)} \left(\sqrt{\e} v_{\yn} - \frac {  f_1(v)}{\sqrt \e}\ \right)^2 d\yn \ \le 
C\left( \int_{-\rho(s)}^{\rho(s)} e_{\e,\nu}(v) d\yn - c_0 \right)+ C e^{-c/\e}
\]
at $y'$.
Integrating this, we find that
\begin{align*}
&\int_{\{ y' \in M: y' \text{ is {\it good}}\}} 
\int_{-\rho(s)}^{\rho(s)} 
 \left(\sqrt{\e} v_{\yn} - \frac {  f_1(v)}{\sqrt \e}\ \right)^2
 d\yn \ d(vol)_0
\nonumber\\
&\qquad\qquad \le 
C\int_{\{y' \in M: y' \text{ is {\it  good}}\}} 
\left( \int_{-\rho(s)}^{\rho(s)} e_{\e,\nu}(v) d\yn - c_0 \right)
\ d(vol)_0
+ C e^{-c/\e}
\nonumber\\
&\qquad\qquad	
\leq C\Big( \zeta_1(s) + c_0 \  vol_0(\{ y' \in M: y'\mbox{ is  {\sl bad}} \})\Big)
 + C e^{-c/\e}.
\end{align*}
Combining these estimates with \eqref{badpts.est}, we conclude that
$I_{1b}\le C (\zeta_1(s) + \zeta_2(s) + \zeta_3(s) + e^{-c/\e})$,
and this together with \eqref{I1split} and \eqref{I.est1} implies \eqref{zeta1.main}.

%

{\bf Step 7.} Having proved \eqref{zeta1.main}, \eqref{zeta2.main}, and \eqref{zeta3.main},
the conclusions of the Proposition follow by a Gr\"onwall inequality argument, 
exactly as in the proof of Proposition 10 in \cite{j-mms}.
%
%
%
%
%
\end{proof}

\section{Proof of Theorem \ref{T.mink}}\label{S:pot3.2}

In this section we prove our main theorem.

\subsection{construction of initial data}
We will prove that the conclusions of the theorem are satisfied by  a smooth solution
$u:N_{T_0}\to \R$ of \eqref{nlwN} with initial data
\[
(u, \partial_{x^0}u)|_{x^0=0} = (u_0, u_1) \quad\mbox{ constructed below}.
\]
To define $u_0,u_1$, we first define $\phi_0: M^{2\rho}\to \Sigma_0$ by requiring that
\[
\phi(0, y', \yn) = (0,\phi_0(y', \yn)).
\]
This definition makes sense in view of \eqref{Phi_b}. 
Next, we set
\beq
u_0 = \sign_{I_0},
\qquad
u_1= 0  \qquad\qquad \mbox{ in }\Sigma_0 \setminus \mbox{Image}(\phi_0),
\label{data.far}\eeq
where
$I_0 := \{ x \in \Sigma_0\cong \R^n : (0,x)\in I\}$.
In Image$(\phi_0)$, it is convenient to specify initial data in term of the 
$y$ coordinates introduced in Proposition \ref{Prop_Phi}.
We would like $v = u \circ \phi$ to satisfy
\beq
v =v_0, \qquad \partial_{y^0}v = 0 \qquad\mbox{ in }M^{2\rho},
\label{v.data}\eeq
when $y^0=0$,
where
\beq
v_0(y', \yn) = \bar q_\e(\yn) :=
\chi_\rho(\yn) q(\frac \yn \e) + (1-\chi_{\rho}(\yn))\sign(\yn)
\label{v0.def}\eeq
and $\chi_\rho\in C^\infty(\R)$ satisfies \eqref{chi.def}.
Thus, we complete the definition of $u_0$ by 
\beq
u_0 = v_0 \circ \phi_0^{-1} \mbox{ in } \mbox{Image}(\phi_0).
\label{u0.near}\eeq
We specify $u_1$ in Image$(\phi_0)$ by requiring that
\[
0 = \partial_{y^0} v = (\partial_{x^\alpha} u\circ \phi) \ \partial_{y^0} \phi^\alpha,
\]
when $y^0 = 0$. Thus we define $u_1 = \partial_{x^0}u|_{x^0=0}$ in Image$(\phi_0)$ 
by the identity
\beq
(u_1\circ \phi) \ \partial_{y^0}\phi^0 = - \sum_{i=1}^n (\partial_{x^i} u_0\circ \phi) \ \partial_{y^0} \phi^i
\qquad\mbox{ in }\{0\} \times M_\rho.
\label{u1.near}\eeq
The construction of $\phi$ implies that $\partial_{y^0}\phi^0$ does not vanish
in $\{0\}\times M_\rho$, so $u_1$ is well-defined,
and \eqref{v.data} holds as desired.

Observe that the definitions 
imply that $u_0$ and $u_1$ are  constant, hence smooth,
near $\partial (\mbox{Image}(\phi_0))$. As a result they are smooth everywhere.

\subsection{first estimates of $v$}\label{firstv}

As remarked above, $v= u \circ \phi$ solves \eqref{nlwM}
(where we are abusing notation and writing $\kappa$ in place of $\kappa\circ \phi$)
with initial data satisfying \eqref{v.data} and coefficients of
the metric tensor satisfying the conclusions of Proposition \ref{Prop_Phi}.
Thus Proposition \ref{localest} is applicable, and in 
particular $v$ satisfies \eqref{Ploc.c1}.
A routine computation, using \eqref{decay} and the rightmost inequality in \eqref{apositive},
shows that for this choice of initial data,
the quantity $\max\big( \zeta_1(0), \zeta_2(0) \big)$ appearing in \eqref{Ploc.c1} satisfies
$
\zeta_0 \le C \e^2$.
In particular it follows that 
\beq
\zeta_3(s) =
\int_{ \{s\}\times  M^{\rho(s)}}
\frac \e 2\sum_{a,b = 0}^{n-1} a^{ab} \partial_{y^a}v\,  \partial_{y^b}v +  (\yn)^2 \left[  \frac \e 2  |\partial_{\yn} v|^2 + \frac 1{\e}F(v) \right] \ d(vol)  \le C \e^2
\label{pf_main1}\eeq
for $0 \le s \le s_1 = \max(T, \rho/(3 c_5))$.

\subsection{short-time estimates of $u$ in the transition region}

Next we define
\[ 
e_\e(u; {\bf h}) 
:= 
\frac \e  2 \left(- h^{00}(\partial_{x^0}u)^2 +
\sum_{i,j=1}^n  h^{ij}
\partial_{x^i}u\, \partial_{x^j}u\right) + \frac 1 \e F(u).
\] 
Note that our assumptions \eqref{lorentzian} on the metric ${\bf h}$
imply that 
\[
- h^{00}(\partial_{x^0}u)^2 +
\sum_{i,j=1}^n  h^{ij}
\partial_{x^i}u\, \partial_{x^j}u
 \ \approx |Du|^2 := \sum_{\alpha=0}^n( \partial_{x^\alpha}u)^2
\]
in $N$,
where $A \approx B$ means that there exists some constant $C$
such that $C^{-1}B \le A \le CB$ pointwise. One can also check
that 
\beq
e_\e(u; {\bf h})\circ \phi \approx e_\e(v ; {\bf g}) 
\qquad\quad\mbox{ in $M_T^{2\rho}$.} 
\label{comparable}\eeq

We next define a smooth cutoff function $\chi^u:N_T\to \R$ such that
$\chi^u = 1$ outside of $N' = \mbox{Image}(\phi)$, and in
$N'$, we require that
\[
\chi^u \circ \phi(y) =  1 - \chi_\rho(\yn),\qquad \mbox{ where $\chi_\rho:\R\to \R$ is defined in \eqref{chi.def}.}
\]
Then $\chi^u$ is smooth,  and $\chi^u = 0$ near $\Gamma$. In particular, 
$\chi^u \circ \phi(y) = 0$ if $|\yn| \le \frac 13 \rho$, and $\nabla\chi^u \circ \phi(y) = 0$
unless $\frac 13 \rho \le |\yn| \le \frac 23 \rho$.
Now fix $t_1>0$ such that
\beq\label{t1.choice}
\{ x\in N : 0 \le x^0 \le t_1 ,  \nabla\chi^u(x) \ne 0\}
\subset \left\{ \phi(y) : y\in M_T^{2\rho}, 0 \le y^0  \le s_1, \frac 13 \rho  \le |\yn | \le \frac 23 \rho\right\}.
\eeq

\begin{center}
\begin{figure}[!htbp]
\includegraphics*[width=350pt]{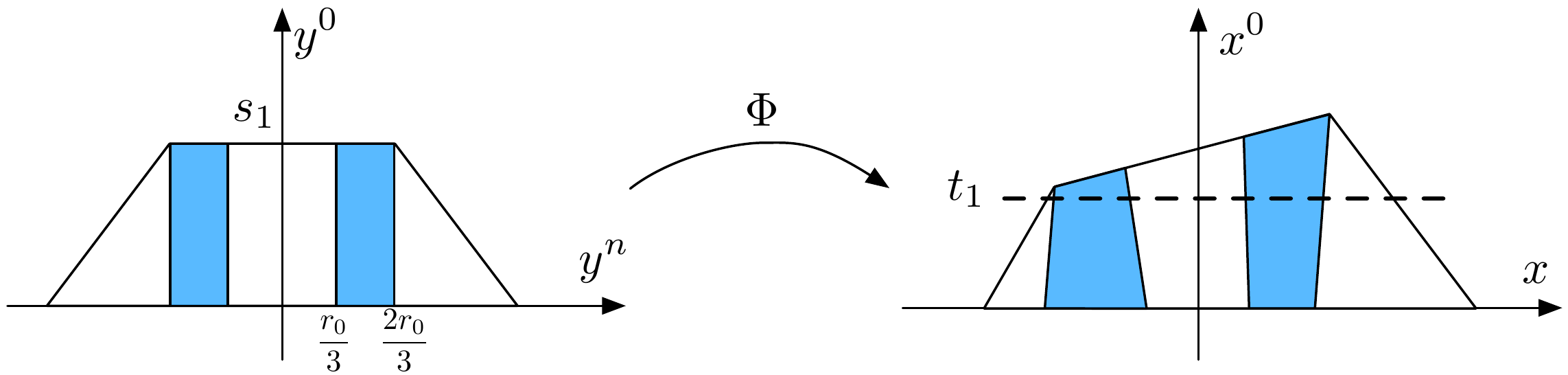}
\caption{}
\label{fig.t1_choice}
\end{figure}
\end{center}

Using \eqref{comparable} and \eqref{pf_main1}, we deduce that
\begin{align}
\int_{ \{x\in N : 0 < x^0 < t_1,  D\chi^u \ne 0\} }
e_\e(u;{\bf h})\, dx\ 
&\ \le \ 
C\int_{ \{ y\in M_T^{2\rho} \ :  0 \le y^0  \le s_1, \frac 13 \rho  \le |\yn | \le \frac 23\ } e_\e(v; {\bf g} )\  d(vol) 
\nonumber\\
&\le C \int_0^{s_1} \zeta_3(s) \, ds\nonumber\\
&\le C \e^2.
\label{shoulder}\end{align}
Here and below, we use Remark \ref{rem:uniform} to obtain uniform
bounds on the Jacobian arising in the change of variables.

\subsection{energy estimates for $u$}

Computations very similar to those in the proof of Lemma \ref{L.eflux}
show that
\[
\frac \partial {\partial x^0} e_\e(u;{\bf h}) 
-  
\sum_{i,j=1}^n\partial_{x^i} ( \e  h^{ij} \partial_{x^j}u \, \partial_{x^0}u)
\ = \ 
\frac \e 2 (\partial_{x^0}h^{\alpha\beta}) \partial_{x^\alpha}u\,\partial_{x^\beta}u +
\e b^\beta \, \partial_{x^\beta}u \ \partial_{x^0}u - \kappa f_1(u) \, \partial_{x^0}u
%
\]
where 
\[
b^\beta := \frac 1{\sqrt{-h}}\  h^{\alpha\beta} \partial_{x^\alpha}\sqrt{-h}.
\]
Since $f_1^2(u) = 2F(u)$ by definition, it easily follows that
\beq
\frac \partial {\partial x^0} e_\e(u;{\bf h}) 
\ \le \   
\sum_{i,j=1}^n\partial_{x^i} ( h^{ij} \partial_{x^j}u \, \partial_{x^0}u) + C e_\e(u,{\bf h}).
\label{weh}\eeq
For $0< \tau<t_1$, we deduce that
\begin{align*}
\zeta(\tau) 
&:= 
\int_{\Sigma_\tau} \chi^u e_\e(u;{\bf h}) \\
&\overset{\eqref{weh}}\le
\zeta(0)+  \int_0^\tau \int_{\Sigma_t} \partial_{x^0} \chi^u  \ e_\e(u;{\bf h}) +  
\chi^u  \left[ \sum_{i,j=1}^n\partial_{x^i} ( h^{ij} \partial_{x^j}u \, \partial_{x^0}u) + C e_\e(u,{\bf h})\right]
\\
&\le
\zeta(0) + C \int_0^\tau \int_{\Sigma_t} \chi^u e_\e(u,{\bf h})\ +  \|D\chi^u \|_\infty\int_{\{x : 0<x^0 <t, |D\chi^u|(x)\ne 0\}}e_\e(u;{\bf h}) 
\\
&\le
C \e^2 + C \int_0^\tau \zeta(t) \ dt.
\end{align*}
In the last line we have used \eqref{shoulder} and an estimate of $\zeta(0)$ that
follows easily from our choice of initial data. The integration by parts
is justified since $u(t, \cdot)=-1$ outside a compact set.
Hence we can conclude by Gr\"onwall's inequality that
\beq
\zeta(t) = \int_{\Sigma_t} \chi^u \ e_\e(u;{\bf h}) \ \le \ C \e^2 \qquad \mbox{ for all }0\le t \le t_1.
\label{uuuuu}\eeq

\subsection{iterate}

We now define $\chi^v:M_T^{2\rho}\to \R$ of the form
$\chi^v(y^0,y', \yn) = \bar \chi(|\yn|)$, where
\[
\bar\chi\in C^\infty(\R),\qquad 0\le \bar\chi \le 1, \qquad \bar\chi = 1\mbox{ in }(\frac 12 \rho, \rho),
\qquad
\mbox{supp}(\bar \chi) \subset (\frac \rho 3 , 2\rho). 
\]
By arguing as in the proof of \eqref{shoulder}, we deduce from
\eqref{uuuuu} find that there exists some $s_0 \in (0, s_1)$ such that 
\[
\int_{ \{ y\in M_T^{2\rho} : 0< y^0 <s_0,  \rho < |\yn| \le 2\rho \} }
 e_\e(v  {\bf;g}) d(vol)  \ \le \  C \e^2.
\]
By combining this with \eqref{pf_main1}, we find that
\[
\int_{ \{ y\in M_T^{2\rho} : 0< y^0 <s_0,   D\chi^v \ne 0 \} }
 e_\e(v  {\bf;g}) d(vol) \ \le \  C \e^2.
\]
Then by using a differential energy inequality satisfied by $v$, see \eqref{weight1}, and
arguing as in the proof of \eqref{uuuuu}, we find that
\[
\int_{ \{s_0\}\times M \times \{ \yn : \frac \rho 2 \le |\yn|\le \rho \}}
 e_\e(v  {\bf;g}) d(vol) \le C \e^2.
\]
Recalling from Section \ref{firstv} that $\zeta_1(s_0) \le C \e^2$ we deduce that $\zeta_1(0;s_0) \le C \e^2$, for 
\beq
\zeta_1(s; s_0) 
:= 
\ 
\left. \int_{M^{\rho(s)}} (1+c_2(\yn)^2)\, e_\e(v) d(vol) \right|_{y^0=(s+s_0)} - c_0 vol_0(M) \ 
\label{vvv}\eeq
We also know from Section \ref{firstv} 
that 
\beq
\zeta_2(0; s_0) \le C \e^2,\quad\qquad
\mbox{ for } \ \zeta_2(s;s_0) := \zeta_2(s-s_0).
\label{vvvv}\eeq

Now we have shown that $v|_{y^0 = s^0}$ satisfies
estimates of the same form (though with larger constants) as $v|_{y^0 = 0}$.
We can thus iterate the above argument to extend estimates first of $v$,
then of $u$,
to somewhat longer time intervals. 
We claim that after piecing together finitely many iterations,
we can obtain the estimates
\begin{align}
\int_{N_{T_0}\setminus N'}
e_\e(u; {\bf h}) \ dx 
\ &\le \ 
C \e^2,
\label{final1}\\
\int_0^T \left(\int_{M^{2\rho} } (1+c_2(\yn)^2) e_\e(v ; {\bf g}) d(vol) - c_0 \, vol_0(M) \right)dy^0 \ 
\ &\le \ 
C \e^2,
\label{final2}\\
\int_0^T \int_{M^{2\rho} }  
\frac \e 2\sum_{a,b = 0}^{n-1} a^{ab} v_{y^a} v_{y^b} +  (\yn)^2 \left[  \frac \e 2  |\partial_{\yn} v|^2 + \frac 1{\e}F(v) \right] \ d(vol)
\ &\le \ 
C \e^2.
\label{final3}
\end{align}
A proof that finitely many iterations suffice
is given in 
in \cite[proof of Theorem 22]{j-mms} for $\kappa = 0$
and in flat Minkowski space, but exactly the same proof is valid here. 
The point is that the proof only involves piecing together estimates in the standard and normal
coordinate systems, and the algorithm for doing so applies equally 
in this situation, since the $(h_{\alpha\beta})$ is uniformly comparable
to the Minkowski metric, see \eqref{lorentzian}.

\subsection{conclusion of proof}

Now \eqref{weightedenergy} follows from \eqref{final1} and \eqref{final3},
together with \eqref{approx.d}, \eqref{comparable} and a change of variables.

The other conclusion of Theorem \ref{T.mink}, that is the estimate \eqref{u.L2} of
$\| u - \sign_I\|_{L^2}$, follows from \eqref{final1}, \eqref{final3}
and a Poincar\'e inequality. We omit the details, which
are just a slightly more complicated
version of the argument used to deduce \eqref{u.L2}
from \eqref{eta1est}  (via \eqref{L2forgooddata}) 
and \eqref{u.energy1}  in the simple
model problem in Section \ref{S:simple}.

\section{Proof of Proposition \ref{Prop_Phi}} \label{S:Phi}

In this section we construct a map $\phi: M_T^{2\rho}\to N$ with the
properties summarized in Proposition \ref{Prop_Phi}.

\subsection{construction of $\phi$}
To start, for $\yt\in M_T$ we define $\bar \nu(\yt)\in T_{\Psi(\yt)}\Gamma \subset T_{\Psi(\yt)}N$
to be the unit normal to $\Gamma$, so that
\beq\label{nu.char}
{\bf h}(\bar \nu , \bar \nu) = 1, \qquad {\bf h}( \bar \nu, \tau) = 0\mbox{ for all }\tau\in T_{\Psi(\yt)}\Gamma.
\eeq
We will fix a sign by requiring that $\bar \nu$ point ``into $I$", see  below. These conditions uniquely determine $\bar\nu$,
and our assumptions imply that $\yt \mapsto \bar\nu(\yt)$ is smooth.
We next define $\wt \phi:M_{T^*}^{2\rho} \to N$,
for $\rho>0$, 
by
\beq
\wt \phi(\yt, \yn)  := \exp_{\Psi(\yt)} (\yn \bar\nu(\yt)).
\label{Phi.def}\eeq
The condition that $\bar\nu$ point into $I$ means that 
$\wt \phi(\yt, \yn)\in I$ for all sufficiently small $\yn>0$.
Thus $\wt\phi$ exactly determine Gaussian normal coordinates 
for $N$ near $\Gamma$.

We will sometimes write $(\wt\phi^0,\ldots, \wt\phi^n)$ to denote the components
of $\wt\phi$ in the fixed coordinate system on $N$. Definition \eqref{Phi.def}
states that the components satisfy the system of 
differential equations
\beq
(\frac{\partial}{\partial \yn})^2 \wt\phi^\alpha + \Gamma^\alpha_{\mu\nu} 
\frac{\partial \wt\phi^\mu}{\partial \yn}
\frac{\partial \wt\phi^\nu}{\partial \yn} = 0,
\qquad
\wt\phi^\alpha(\yt, 0) = \psi^\alpha(\yt), 
\quad
\frac{\partial \wt\phi^\alpha}{\partial \yn}(\yt,0) = \bar\nu^\alpha(\yt),
\label{geodesics}\eeq
where $\bar\nu^\alpha$, $\alpha = 0,\ldots, n$ are the components of $\bar\nu$
and 
\beq
\Gamma^\alpha_{\mu\nu} =
\frac 12 h^{\alpha\beta}\left(
\frac{\partial }{\partial y^\mu} h_{\beta\nu}
+
\frac{\partial }{\partial y^\nu} h_{\mu\beta}
-
\frac{\partial }{\partial y^\beta} h_{\nu\mu}
\right)
\label{christoffel}\eeq
are the usual Christoffel symbols.

Finally, we define 
\[
\phi(y^0,\ldots, y^n) :=  \wt \phi(y^0 - \sigma(y', y^n), y^1,\ldots, y^n)
\]
where $\sigma:M^{2\rho}\to \R$ is chosen exactly so that
\eqref{Phi_b} holds. Thus, we require that $\sigma$ satisfy
\beq
\wt\phi^0( - \sigma(y', \yn) , y', \yn) = 0
\quad\mbox{ in }M^{2\rho}.
\label{sigma.def}\eeq
The next lemma implies that the definition of $\phi$ makes sense.

\begin{lemma}
For $\rho$ sufficiently small, there exists  $\sigma:M_\rho\to \R$ satisfying 
\eqref{sigma.def} and in addition
\beq
|\sigma(y',\yn) | \le C (\yn)^2.
\label{sigma.quadratic}\eeq
\end{lemma}

\begin{proof}
The definitions \eqref{Phi.def} and \eqref{Psi1} 
of $\wt \phi$ and $\Psi$ implies that $\wt\phi^0(\yt, 0) = y^0$.
Thus it is clear that 
$\partial_{y^0}\wt \phi^0(\yt, 0) = 1$.
Since $M$ is compact,
the implicit function theorem thus implies that, taking $\rho$
smaller if necessary, there exists
a function $\sigma :M^{2\rho} \to \R$ such that
\eqref{sigma.def} holds. 
To prove \eqref{sigma.quadratic}, it suffices (again
using the compactness of $M$) to show that
\beq
\sigma(y',0) = 0 , \qquad  \frac{\partial \sigma}{\partial y^n}(y', 0)= 0 \qquad\mbox{ for every }y'\in M.
\label{sigquad}\eeq
The first of these assertions is clear.
For the second, we differentiate \eqref{sigma.def}
with respect to $y^n$ and evaluate at a point
$(y', 0)$, to find that
\[
\frac{\partial\wt  \phi^0}{\partial y^n}(0,y', 0)
=
\frac {\partial \wt \phi^0}{\partial y^0}(0,y', 0) \frac{\partial \sigma}{\partial y^n}(y',0)
=  \frac{\partial \sigma}{\partial y^n}(y', 0).
\]
So in view of \eqref{geodesics}, to complete the proof of \eqref{sigquad}
it suffices to prove that $\bar \nu^0(0,y') = 0$ for all $y'\in M$.
But this follows from noting that at points of the form $(0,y')$,
\[
0 \ = \ 
{\bf h}( \bar \nu , \frac{\partial \Psi}{\partial y^0})
\ = \ 
h_{\alpha\beta}\, \bar \nu^\alpha \, \frac{\partial \Psi^\beta}{\partial y^0}
\ \overset{\eqref{Psi1},\eqref{zv2}}
{\ = \ } \ 
h_{\alpha 0}\, \bar \nu^\alpha   
\ \overset{\eqref{hsplit}}
{\ = \ } \ 
h_{00}\bar \nu^0.
\]
\end{proof}

Our arguments will imply that $\phi$ is locally invertible in a neighborhood of
every point of $M_{T^*}$, and it follows from this that for every
$T\in (T_0, T^*)$ there exists $\rho>0$ such that $\phi$
is a diffeomorphism of $M_T^{2\rho}$ onto its image. 
We henceforth assume that this holds. 
We will also feel free to decrease the size of $\rho$ throughout
our argument.

\subsection{estimates of components of the metric tensor}

We next prove \eqref{glowerab}, \eqref{gupperab}, and \eqref{dtgupperab}.
We will use the notation
\[
\wt g_{\alpha\beta} := {\bf h}( 
\frac {\partial \wt \phi}{\partial y^\alpha},
\frac {\partial \wt \phi}{\partial y^\beta}),
\qquad
(\wt g^{\alpha\beta}) := (\wt g_{\alpha\beta})^{-1}, \qquad 
\wt g:= \det (\wt g_{\alpha\beta})
\]
for $\alpha,\beta=0,\ldots, n$.
We first remark that
\beq
(\wt g_{\alpha\beta})(\yt, \yn) = 
\left(
\begin{array}{ll}
(\gamma_{ab})(\yt) &0\\
0& \mbox{1}
\end{array}\right) +
\left(
\begin{array}{ll}
O(|\yn|) &0 \\
0& 0
\end{array}\right) 
\label{tglowerab}\eeq
(in block  form), where $(\gamma_{ab})$ was introduced in
\eqref{gammaab}.
The estimate $\wt  g_{ab}(\yt,\yn) = \gamma_{ab}(\yt) + O(\yn)$
for $a,b<n$ is immediate when $\yn=0$, since $\wt\phi(\yt,0) = \Psi(\yt)$, and then follows by 
the smoothness of $\wt \phi$. The claim above that
$\wt g_{\alpha n} = \wt  g_{n,\alpha} = \delta_{n\alpha}$ for all $(\yt, \yn)$
is standard; see for example Wald \cite{Wald} Section 3.3.

It is convenient to define $\widehat \phi(y^0,\ldots, \yn) := 
(y^0 - \sigma(y', \yn), y^1,\ldots, \yn)$, so that 
$\phi = \wt \phi \circ\widehat \phi$. Then the definitions
imply that
\beq \label{39} 
g_{\alpha\beta}(y) 
\ = \ 
\wt  g_{\mu\nu}(\hat \phi(y)) 
\frac {\partial  \widehat\phi^\mu}{\partial y^\alpha}(y)\ 
\frac {\partial \widehat\phi^\nu}{\partial y^\beta}(y) .
\eeq 
It follows from \eqref{sigma.quadratic} that $g_{\mu\nu}(\hat \phi(y))  = g_{\mu\nu}(y) + O((\yn)^2)$
and that $\frac {\partial \sigma}{\partial y^i} = O(|\yn|)$  for $i\ge 1$
and from \eqref{Psi2} and  \eqref{tglowerab} that 
$\wt g_{0i} = \wt g_{i0} = O(|\yn|)$ for $i\ge 1$.
From these one can check that
\[
\wt g_{\alpha\beta}(y) = g_{\alpha\beta}(y) +O((\yn)^2).
\]
It follows from this and \eqref{tglowerab} that
\beq
( g_{\alpha\beta})(\yt, \yn) = 
\left(
\begin{array}{ll}
(\gamma_{ab})(\yt) &0\\
0& \mbox{1}
\end{array}\right) +
\left(
\begin{array}{ll}
O(|\yn|) &O(|\yn|) \\
O(|\yn|)& O((\yn)^2)
\end{array}\right) 
\label{glowerab.bis}\eeq
which is \eqref{glowerab}.
Then \eqref{gupperab} follows by elementary linear algebra considerations;
see for example Lemma 26 in \cite{j-mms}.

We must also estimate $\partial_0g^{\alpha\beta}$.
To do this, we follow
\cite{j-mms} and 
differentiate the identity $g^{\alpha\beta} g_{\beta\gamma} = \delta^\alpha_\gamma$
and rearrange to find that
\[
\partial_{y^0} g^{\alpha\beta} = -g^{\alpha\mu} \ \partial_0 g_{\mu\nu}\  g^{\nu\beta}.
\]
Next, after differentiating \eqref{39} and using \eqref{tglowerab},
some calculations show that 
\[
(\partial_0 g_{\mu\nu}) = \left(
\begin{array}{ll}
O(1) &O(|\yn|) \\
O(|\yn|)& O((\yn)^2)
\end{array}\right),
\]
and then \eqref{dtgupperab} can be deduced from this and the estimates  \eqref{gupperab}
for $(g^{\alpha\beta})$ found above.

\subsection{ mean curvature in almost-normal coordinates}\label{ss:mc}

We next prove \eqref{g3},
which states that
\beq
\frac {-1}{\sqrt{-g}} g^{n\alpha} \partial_{y^\alpha}\sqrt{-g} = - \kappa(\yt)
+O(|\yn|).
\label{bbis}\eeq
where the left-hand side is evaluated at
$(\yt,\yn)$ and $-\kappa(\yn)$ is the mean curvature of $\Gamma$ at $\phi(\yn)$,
since $\Gamma$ is assumed to satisfy \eqref{pmcm}. Since the left-hand side
is smooth, it suffices to check this when $\yn=0$.

To do this, we simply verify that the left-hand side of
\eqref{bbis} equals the mean curvature of $\Gamma$ with
respect to the $y$ coordinates. In these coordinates, $\Gamma$ is
simply the set $\{y\in M_T^{2\rho} : \yn = 0\} = M_T$,
and the unit normal to $\Gamma$  by construction is  the
vector  $ \frac{\partial}{\partial \yn}$.
In general, the mean curvature is the
trace of the second 
fundamental form $A$, 
where $A:T_y\Gamma\times T_y\Gamma \to \R$ is defined by 
\[
A(X,Y) := {\bf g}( \nu, \nabla_X Y ).
\]
Thus
\[
\mbox{mean curvature } \ = \   
g^{\alpha\beta} A( 
\frac {\partial }{\partial y^\alpha} , 
\frac {\partial }{\partial y^\beta} )
\ = \ 
g^{\alpha\beta} {\bf g}( \frac{\partial}{\partial y^n} , \nabla_{ 
\frac {\partial }{\partial y^\alpha} }
\frac {\partial }{\partial y^\beta} )
\ = \ g^{\alpha\beta} \Gamma^n_{\alpha\beta}.
\]
And from the standard expression for
the Christoffel symbols, see \eqref{christoffel},
and using \eqref{glowerab} with $\yn=0$, we compute that
\beq
g^{\alpha\beta} \Gamma^n_{\alpha\beta}
\ = \ -\frac 12 g^{\alpha\beta}\partial_{y^n} g_{\alpha\beta}.
\label{mc.nc}\eeq
On the other hand, a standard computation shows that
\[
\frac {-1}{\sqrt{-g}} g^{n\alpha} \partial_{y^\alpha}\sqrt{-g} =
\, -\,\frac 12 g^{n\alpha} g^{\mu\nu} \partial_{y^\alpha} g_{\mu\nu},
\]
and by \eqref{gupperab},
this agrees with the right-hand side of \eqref{mc.nc} when $\yn=0$.

\begin{remark}
The above proof shows that, in addition to \eqref{tglowerab}, we have
\[
(\wt g^{\alpha\beta})(\yt, \yn) = 
\left(
\begin{array}{ll}
(\gamma^{ab})(\yt) &0\\
0& \mbox{1}
\end{array}\right) +
\left(
\begin{array}{ll}
O(|\yn|) &0 \\
0& 0
\end{array}\right) ,
\]
\[(\partial_0 \wt g^{\alpha\beta})(\yt, \yn) = 
\left(
\begin{array}{ll}
O(1) &0\\
0& 0
\end{array}\right),
\]
and 
\[
\frac{- 1}{\sqrt{-\wt g}} \wt g^{n\alpha}\partial_{y^\alpha} \sqrt{-\wt g}  \ = \ 
\frac {-1}{\sqrt{-\wt g}} \partial_{\yn} \sqrt{-\wt g}  \ = \ 
-\kappa(\yt) + O(|\yn|).
\]
In particular the first two of these facts, together with \eqref{tglowerab}, are
somewhat better than the corresponding properties of $(g_{\alpha\beta})$.
\end{remark}

\subsection{solution of the eikonal equation}
Finally, it is well-known that if we define $d_\Gamma := \pi^n \circ \wt \phi^{-1}$,
where $\pi^n(y^0,\ldots, y^n) = y^n$,
then $d_\Gamma$ satisfies \eqref{dGamma} in $\mbox{Image}(\phi)$.
Indeed, the definitions imply that $d_\Gamma = 0$ on $\Gamma$,
and we have fixed signs such that $d_\Gamma>0$ in $\mbox{Image}(\phi)\cap I$.
Finally, by inverting the definition
\[
\wt g_{\alpha\beta} =  h_{\mu\nu}\ \partial_{y^\alpha}\wt \phi^\mu\  \partial_{y^\beta}\wt\phi^\nu 
\]
and expressing $(D\wt\phi)^{-1}$ in terms of $D (\wt\phi^{-1})\circ \phi$, we find that
\[
\wt g^{\alpha\beta}\circ \phi^{-1} = h^{\mu\nu} \partial_{x^\mu}(\wt\phi^{-1})^\alpha 
 \partial_{x^\nu}(\wt\phi^{-1})^\beta .
\]
Thus the fact that $\tilde g^{nn} = 1$ states exactly that $h^{\alpha\beta}\partial_{x^\alpha}d_\Gamma \ \partial_{x^\beta} d_\Gamma =1$.

Finally, we deduce from \eqref{sigma.quadratic} that
$\phi = \wt \phi$ and $D\phi = D \wt \phi$ everywhere on $\{ y^n=0\}$,
and it follows from the inverse function theorem that $\phi^{-1} = \wt\phi^{-1},
D(\phi^{-1}) = D(\wt \phi^{-1})$ everywhere on $\Gamma$.
This implies that $\phi^{-1} - \wt \phi^{-1} = O(d_\Gamma^2)$,
and since $d_\Gamma = (\wt\phi^{-1})^n$, 
we arrive at \eqref{approx.d}.

\end{document}